\documentclass[10pt, leqno]{amsart}

\usepackage{amsmath}
\usepackage{dsfont}\usepackage{color}
\usepackage{graphicx}\usepackage{amsbsy}
\usepackage{verbatim}
\usepackage{amsthm}

\renewcommand{\epsilon}{\varepsilon}
\DeclareMathOperator{\dvg}{div} \DeclareMathOperator{\spt}{spt}
\DeclareMathOperator{\dist}{dist} 
\DeclareMathOperator{\graph}{graph} 
\DeclareMathOperator{\trace}{trace}\DeclareMathOperator{\diam}{diam}
  \DeclareMathOperator{\Nor}{Nor}
  \DeclareMathOperator{\proj}{proj}
  \DeclareMathOperator{\refl}{ref}
  
\title[$C^{1,\alpha}$ theory for the MCE with Dirichlet data]
{$C^{1,\alpha}$ theory for the prescribed mean curvature equation with Dirichlet data}

\author{Theodora Bourni}

\newtheoremstyle{break}
 {9pt}
 {9pt}
 {\itshape}
 {}
 {\bfseries}
 {.}
 {\newline}
 {}
\theoremstyle{plain}
\newtheorem{theorem}[equation]{\textbf{Theorem}}
\newtheorem{lemma}[equation]{\textbf{Lemma}}
\newtheorem{definition}[equation]{\textbf{Definition}}

\newtheorem{rmk}[equation]{\textbf{Remark}}
\newtheorem{cor}[equation]{\textbf{Corollary}}

\numberwithin{equation}{section}

\theoremstyle{break}

\makeatletter
\def\tagform@#1{\maketag@@@{\ignorespaces#1\unskip\@@italiccorr}}
\makeatother

\makeatletter
\def\subsection{\@startsection {subsection}{2}{\z@}{4ex}{1.5ex}{\normalsize\bf}}
\makeatother

\def\R{\mathbb{R}}

\def\W{\Omega}
\def\d{\delta}

\def\a{\alpha}
\def\e{\epsilon}

\def\r{\rho}

\def\s{\sigma}
\def\o{\omega}
\def\l{\lambda}
\def\k{\kappa}

\def\H{\mathcal{H}}

\def\mass{\underline{\underline{M}}}
\def\wt{\widetilde}

\def\ov{\overline}
\def\res{\hbox{ {\vrule height .25cm}{\leaders\hrule\hskip.2cm}}\hskip5.0\mu}
\def\B{\mathcal{B}}

\def\BV{\text{BV}}
\def\toa{\stackrel{C^{1,\a}}{\longrightarrow}}
\def\toap{\stackrel{C^{1,\a'}}{\longrightarrow}}

\begin{document}
\maketitle
\begin{abstract}
 In this work we study solutions of the prescribed mean curvature equation over a general domain that do not necessarily attain the given boundary data. To such a solution, we can naturally associate a current with support in the closed cylinder above the domain and with boundary given by the prescribed boundary data and which  inherits a natural minimizing property. Our main result is that its support is a $C^{1,\alpha}$ manifold-with-boundary, with boundary equal to the prescribed boundary data, provided that both the initial domain and the prescribed boundary data are of class $C^{1,\alpha}$.\end{abstract}

\section{Introduction}
The Dirichlet problem for surfaces of prescribed mean
curvature in an open set $\Omega$ of $\mathbb{R}^n$ concerns the existence of a solution
to the equation
\begin{equation}\label{PMCE}
  \sum_{i=1}^nD_i\left(\frac{D_iu}{\sqrt{1+|Du| ^2}}\right)=H(x,u)\hbox{  in  }\Omega
  \end{equation}
  taking prescribed values 
  \[u=\phi\text{   on   }\partial\Omega.\] 
    Here and throughout this paper $\W\subset\R^n$ is an open bounded set, $\phi\in L^1(\partial\W)$ and $H(x,x_{n+1})$ is a $C^1$ function defined in $\W\times\R$, which is non-decreasing in the $x_{n+1}$-variable and such that $\|H\|_{0}\le n\left(\o_n/|\Omega|\right)^{1/n}$.
  
  It is known \cite{JS-MSE-Dirichletproblem, SPMCE-andotherDirichletproblems} that if $\partial\W$ is $C^2$, then a solution exists for any given boundary values $\phi\in C^0(\partial\W)$ provided that $H_{\partial\W}(x)>|H(x,\phi(x))|$ for each $x\in\partial\W$, where $H_{\partial\W}$ denotes the mean curvature of the boundary and furthermore the regularity of the solution depends on that of $\partial\W$ and $\phi$. Here and in what follows we adopt the sign convention according to which the mean curvature of $\partial\W$ is non-negative in case $\W$ is convex. Furthermore, there are examples that indicate that this condition is necessary for the existence of a solution (cf.\cite[14.4]{GT}).
  
 Our goal is to study the regularity of such a solution without imposing any curvature conditions for $\partial\W$. For this reason we will use a variational approach to the Dirichlet problem (cf. \cite{Giu-Introof-DeGiorgi-functional, MM-DeGiorgi-functionalwithH}) and look for a minimum of the functional
  \begin{equation}\label{functional F}
\mathcal{F}(v)=\int_\Omega\sqrt{1+|Dv|^2}dx+\int_\Omega\int_0^{v(x)}H(x,x_{n+1})dxdx_{n+1}+\int_{\partial
  \Omega}|v-\phi|dx
  \end{equation}
for $v\in \BV(\W)$; here $\BV(\Omega)$ denotes the space of all functions
  in $L^1(\Omega)$ that have bounded variation, i.e. with first distribution derivatives given by signed Radon measures. 
  
  Giusti and Miranda
\cite{Giu-Introof-DeGiorgi-functional, MM-DeGiorgi-functionalwithH} have proved that if $\partial\Omega$ is Lipschitz, then there exists a minimizer $u$ of the functional $\mathcal{F}$, which is unique up to translations. Furthermore this minimizer satisfies equation \ref{PMCE} in $\W$ (cf. \cite{BomGiusti, LUgradbounds, TrudGradest}) and attains the prescribed boundary values above any $C^2$ portion of the boundary where the mean curvature is bigger than $|H(x,\phi(x))|$ \cite{MM-lpcdomaind-weakMSE}. 
  
  The purpose of this paper is to give a complete and general discussion on the regularity of the hypersurface obtained by taking the union of $\{(x,u(x)): x\in\ov{\W}\}$ and the part of $\partial\W\times\R$ which is enclosed by $\{(x,u(x)):x\in\partial\W\}$ and $\{(x,\phi(x)):x\in\partial\W\}$, where $u$ is the minimizer of $\mathcal{F}$.  In particular in our Main Theorem (Theorem \ref{main theorem}) we prove that if $\partial\W$ is $C^{1,\a}$ and $\phi\in C^{1,\a}(\partial\W)$, then this hypersurface is a $C^{1,\a}$ manifold-with-boundary, with boundary equal to $\graph\phi$. We also show that this regularity result can be extended for boundary data $\phi\in C^{1,a}(\partial\W\setminus\{x_0\})$, where at $x_0\in\partial\W$, $\phi$ has a jump discontinuity.
  
  Furthermore we will show that this manifold can be obtained as the $C^{1,\a}$ limit (as submanifolds of $\R^{n+1}$) of graphs of $C^{1,\a}$ functions over $\ov\W$. The main idea is to approximate the equation of the given Dirichlet problem \ref{PMCE} by new equations in which we change the RHS near the boundary by adding a divergence term that will allow us to prove existence of barriers for solutions of the new equations. We then use techniques from the theory of integer multiplicity varifolds, integral currents and partial differential equations to get uniform $C^{1,\a}$ estimates for the graphs of the solutions to the approximating equations.

 Concerning the regularity of $u$ (the minimizer of $\mathcal{F}$), it is known \cite{Giu-Bdry-behavior-ofnonparametric-MS, SimonHolderContinuity} that if $\phi$ is Lipschitz, then above any $C^2$ portion of the boundary where the mean curvature is bigger than $|H(x,\phi(x))|$, $u$ is Holder continuous for some positive exponent.
  However above points of the boundary where this condition is not satisfied, we could have $u\ne\phi$ and
   there are examples that show that the gradient of $u$ does not have to
   be bounded near these points. In \cite{LS-bdry-regularity-onnonparametric-MS} it is proved that if $\Omega$ is a $C^4$
domain and $\phi$ is a Lipschitz function over $\partial\W$ then in the case $H=0$, $u$ is Holder continuous
at every point
  $x\in \partial\Omega$ where the mean curvature is negative and furthermore the trace of $u$ as a function
   above $\partial\Omega$
   is locally Lipschitz at these points. Note that since $u\in\BV(\W)$ it has a well defined trace in $ L^1(\partial\W)$. In \cite{FHLinNonparametric} this result was extended for surfaces of prescribed mean curvature $H=H(x)$.
 
 The hypersurface that corresponds to $u$, as described above, inherits a minimizing property which we now describe: 
 To a function $v\in \BV(\W)$ we can associate an integral $n$-current defined by
  \begin{equation}\label{current T}
  T_v=[\![\graph v]\!]+Q
  \end{equation}
  where $Q$ is the multiplicity 1 $n$-current with support in $\partial\W\times\R$ and boundary $\partial Q=[\![\graph\phi]\!]-[\![\trace v]\!]$. Here and in what follows for the orientation of a current $[\![\graph v]\!]$ associated to the graph of a function $v$ we use the downward pointing unit normal to the graph. For any multiplicity 1 $n$-current $S$ such that $\spt S\subset\ov{\W}\times\R$ and $\partial S=[\![\graph\phi]\!]$ we let $\wt{S}$ be the multiplicity 1 $(n+1)$-current such that $S-[\![(x,z):x\in\partial\W, z\le\phi(x)]\!]=\partial\wt{S}$. Then if $u$ minimizes the functional $\mathcal{F}$, the current $T=T_u$, as defined in \ref{current T}, locally minimizes the functional
  \begin{equation}\label{minimizing T-property}
  \mass(T)+\int_{{\spt}\widetilde{T}}H(x',x_{n+1})dx'dx_{n+1}
  \end{equation}
  among all integral $n$-currents with support in $\ov{\W}\times\R$ and boundary $[\![\graph\phi]\!]$
  \cite{FHLinNonparametric}, where $\mass(T)$ denotes the mass of the current $T$.
 
This observation was first made by Lin and Lau \cite{Lin-Lau} for the $H=0$ case. In particular they observed that in that case ${T}$ minimizes area among all integral currents with support in $\ov{\W}\times\R$ and boundary equal to $[\![\graph\phi]\!]$, thus locally, near points of the trace of $u$ that are away from $\graph\phi$, $\spt T$ is a solution to a parametric obstacle problem. Hence, using results from \cite{BK-variationalinequalities, MM-frontiereminimali} in case $\W$ is a $C^2$ domain, they showed that $\spt T$ is a $C^{1,1}$ manifold near such points.

  There are various results concerning the regularity of minimal boundaries respecting a given obstacle \cite{MM-frontiereminimali, BarozziMassari, Tamanini}, however these results (as that of Lin and Lau) do not include any discussion about boundary points and hence, using these results, we cannot conclude anything about the regularity around points in the intersection $\trace u\cap\graph\phi$.
 
 Finally we mention that if $\partial\W$ is of class $C^2$ then, following the notation of \cite{DuzaarSteffen-lambda}, the current $T=T_u$ is $\lambda$-minimizing, i.e.
 \[\mass(T)\le\mass(T+\partial Q)+\l\mass(Q)\]
 for all integral $(n+1)$-currents $Q$, where $\l=\max\{\|H\|_0,\|H_{\partial\W}\|_0\}$.
 In \cite{DuzaarSteffen-bdryregularity}, Duzaar and Steffen generalized for such currents the boundary regularity results given in \cite{LSRHbdryregularity} for area minimizing currents. In particular they proved that if $\partial T$ is represented by a multiplicity 1, $C^{1,\a}$ submanifold, then $\spt T$ is a $C^{1,\beta}$ submanifold, for all $\beta\le\a/2$, around each point $a\in\partial T$, where $\Theta_T(a)<1+1/2$.

\section{The Dirichlet problem with regular data}\label{section the DP with RD}
\subsection*{Notation and Definitions}
$B^m_r(x)$ will denote the $m$-dimensional ball of radius $r$ centered at a point $x\in\R^m$, i.e.
\[ B^m_r(x)=\{y\in\R^m:|y-x|<r\}\]
and $\o_m$ will denote the measure of the $m$-dimensional unit ball.

For any $C^{1,\a}$ function $u:V\cap B^m_r(0)\to\R^n$, where $V\subset\R^m$ is a $C^{1,\a}$ domain and $\a\in(0,1]$, $\|u\|_{1,\a, V\cap B^m_r(x)}$ will denote the scaled $C^{1,\a}$ norm of $u$, i.e.:
\[\|u\|_{1,\a, V\cap B^m_r(x)}=\frac{1}{r}\|u\|_0+\|Du\|_0+r^{\a}[D u]_\a\]
 Occasionally, when there is no confusion about the domain of $u$, we will write $\|u\|_{1,\a}$ instead of $\|u\|_{1,\a,V\cap B^m_r(x)}$.

For a point $x\in\R^{n+1}$ we will often write $x=(x', x_{n+1})$, where $x'\in\R^n$. Finally the letter $c$ will denote a constant depending only on the specified parameters and when different constants appear in the course of a proof we will keep the same letter $c$ unless the constant depends on some different parameters.

\begin{definition}\label{c1a convergence}
 $M\subset \R^{n+1}$ is an $m$-dimensional properly
embedded $C^{1,\a}$ submanifold \emph{(}where $m\le n$, $\a\in(0,1]$\emph{)} if for each $x\in M$ there is a $\rho>0$ such that
$$
M\cap B^{n+1}_{\rho}(x) = \graph u_{x}\cap  B^{n+1}_{\rho}(x),
$$
where $u_{x}\in C^{1,\alpha}((x+L_{x})\cap \overline B^{n+1}_{\rho}(x);L_{x}^{\perp})$ for some
$m$-dimensional subspace $L_{x}$ of $\R^{n+1}$ and where $\graph
u_{x}=\{\xi+u_{x}(\xi):\xi\in (x+L_{x})\cap  \ov B^{n+1}_{\rho}(x)\}$.  We quantify the regularity of
$M\cap B^{n+1}_{\rho}(x)$ by defining
\[\kappa(M,\rho,x) = \inf\|u_x\|_{1,\a,(x+L_x)\cap B^{n+1}_\r(x)}\]
where the infimum is taken over all choices of subspaces $L_{x}$ and corresponding representing functions
$u_{x}$.

We say that a sequence $M_{k}$ of $m$-dimensional submanifolds converges in the $C^{1,\a}$ sense to $M$ in $B^{n+1}_{\rho}(x)$ \emph{(}for $\rho>0$ and $x\in M$\emph{)} and write
\[M_k\toa M \text{  in  } B^{n+1}_{\rho}(x)\]
 if there exists a subspace $L_{x}$ and functions $u,u_{k}\in C^{1,\alpha}((x+L_{x})\cap\ov B^{n+1}_{\rho}(x);L_{x}^{\perp})$ with 
\[M_{k}\cap
B^{n+1}_{\rho}(x)=\graph u_{k}\cap B^{n+1}_{\rho}(x)\text{  ,  }M\cap B^{n+1}_{\rho}(x)=\graph u\cap B^{n+1}_{\rho}(x)\] and 
\[\|u_{k}-u\|_{1,\a,(x+L_x)\cap B^{n+1}_\r(x)}\to 0.\]

We then say that $M_{k}$ converges in the $C^{1,\a}$ sense to $M$ in $\R^{n+1}$  and write
 $M_{k}\toa M$, if there is a $\rho>0$ such that $M_{k}\subset\cup_{x\in
M}B^{n+1}_{\rho}(x)$ for all sufficiently large $k$ and if $M_k\toa M$ in $B^{n+1}_\r(x)$  for each $x\in M$.
\end{definition}

\begin{definition}[Regular Class]\label{ar regular class}
For $\a\in(0,1], r>0$ we define the $(\a,r)$-regular class, which we denote by $\B^\a_{r}$, to be the set of all pairs $(\W,\Phi)$ satisfying the following:

\begin{enumerate}
\item[1.] $\W$ is a domain of $\R^n$ such that $\partial\W\cap B^n_{r}(0)$ is a non-empty, $(n-1)$-dimensional embedded $C^{1,a}$ submanifold of $\R^n$ such that
\[\k(\partial\W,r,x)<1\,\,,\,\,\forall x\in\partial\W\cap B^n_r(0).\]
\item[2.] There exists a sequence of functions $\{\phi_i\}\subset L^1(\partial{\W})\cap C^{1,\a}(\partial{\W}\cap B^n_r(0))$ such that $\graph\phi_i\toa \Phi$ in $B^n_r(0)\times\R$ and $\Phi\cap (B^n_r(0)\times\R)$ is an $(n-1)$-dimensional embedded $C^{1,a}$ submanifold of $\R^{n+1}$ such that 
\[\k(\Phi\cap (B^n_r(0)\times\R),r,x)<1\,\,,\,\,\forall x\in\Phi\cap (B^n_r(0)\times\R).\]
\end{enumerate}

For $(\W,\Phi)\in\B^\a_{r}$ we define
\[{\k}_{(\W,\Phi)}=\max\{\sup_{x\in\partial\W\cap B^n_r(0)}\k(\partial\W,r,x), \sup_{x\in\Phi\cap (B^n_r(0)\times\R)}\k(\Phi\cap (B^n_r(0)\times\R),r,x)\}.\]
\end{definition}

\begin{rmk}\label{graphical expression of regular class}
Note that if $(\W,\Phi)\in\B^\a_r$ then
\begin{equation*} 
\|\nu_{\partial\W}(x)-\nu_{\partial\W}(y)\|\le c(n)\k_{(\W,\Phi)}|x-y|^{\a} 
\end{equation*}$\forall \, x,y\in\partial\W\cap B^n_r(0)$,
where for $x\in\partial\W$, $\nu_{\partial\W}(x)$ denotes the inward pointing unit normal to $\partial\W$ at $x$, and
\begin{equation*}
\|\proj_{\Nor_{\Phi}(x)}-\proj_{\Nor_{\Phi}(y)}\|\le c(n)\k_{(\W,\Phi)}|x-y|^{\a}
\end{equation*}
$\forall\, x,y\in\Phi\cap (B^n_r(0)\times\R)$ such that $|x-y|\le r$, where for $x\in\Phi$, $\Nor_{\Phi}(x)$ denotes the 2-dimensional normal subspace to $\Phi$ at $x$.
\end{rmk}

The following remark is a direct consequence of the Arzela-Ascoli theorem.
\begin{rmk}\label{AA remark}
Let $(\W_i,\Phi_i)\in\B^\a_{r_i}$ be a sequence such that $\liminf r_i=\infty$ and for some $r\in(0,\infty)$, $B_{r}^n(0)\cap\partial\W_i\ne\emptyset$ for all $i$. Then 
after passing to a subsequence
\[\W_i\toap \W \text{  and  } \Phi_i\toap \Phi \]
for any $\a'<\a$ in the sense of \emph{Definition \ref{c1a convergence}} and where $(\W,\Phi)\in \B^\a_{r'}$ for all $r'>r$.

If in addition $\k_{(\W_i,\Phi_i)}\to0$, then for the limit we have that $(\W,\Phi)=(H,\Phi)$,
where $H$ is an $n$-dimensional halfspace and $\Phi\subset \partial H\times\R$ is an $(n-1)$-dimensional linear space or $\emptyset$.
\end{rmk}

\subsection*{The Dirichlet Problem}\label{DP subsection}

Let $\W$ be a domain of $\R^n$ and $\phi\in L^{1}(\partial{\W})$ be such that $(\W,\graph \phi )\in {\B}^\a_{r}$, for some $\a\in(0,1]$ and $r>0$. We consider the following Dirichlet problem:

\begin{equation}\label{PMC dirichlet problem with divergence term}
\begin{split}
\sum_{i=1}^nD_i\left(\frac{D_iu}{\sqrt{1+|Du|^2}}\right)&=H+\sum_{i=1}^n D_if^i \text{  in  } \W\\
u&=\phi\text{  on  }\partial\W 
\end{split}
\end{equation}
where $H=H(x, u(x)), f^i=f^i(x,u(x))\in L^1(\ov{\W}\times\R)$, $H$ is bounded in $(\ov{\W}\cap B_r^n(0))\times\R$ and $f=(f^1,\dots f^n)$ is a $C^{0,\a}$ vector field in $(\ov{\W}\cap B_r(0))\times\R$, so that
\[\|H\|_{0,(\ov\W\cap B^n_r(0))\times\R}=\sup_{x\in(\ov\W\cap B^n_r(0))\times\R}|H(x)|<\infty\]
and
\[[f]_{\a,(\ov\W\cap B^n_r(0))\times\R}=\sup_{x,y\in(\ov\W\cap B^n_r(0))\times\R}\frac{|f(x)-f(y)|}{|x-y|}<\infty.\]
For notational simplicity and as long as there is no confusion about the domain $\W$, we will write $\|H\|_{0,B^n_r(0)\times\R}$ and $[f]_{\a, B^n_r(0)\times\R}$ instead of $\|H\|_{0,(\ov\W\cap B^n_r(0))\times\R}$ and $[f]_{\a,(\ov\W\cap B^n_r(0))\times\R}$ respectively.

The equation in \ref{PMC dirichlet problem with divergence term} above is to be interpreted weakly, i.e.
\begin{equation}\label{MCE*}
\int_\W\sum_{i=1}^n\frac{D_iu}{\sqrt{1+|Du|^2}}D_i\zeta d\H^n=-\int_\W H\zeta d\H^n+\int_\W\sum _{i=1}^nf_i D_i\zeta d\H^n
\end{equation}
for any $\zeta\in C^1_c(\W)$.

For the rest of Section \ref{section the DP with RD} we will let $u\in C^{1,\a}(\ov{\W})$ be a (weak) solution of the Dirichlet problem \ref{PMC dirichlet problem with divergence term} and $T=[\![\graph u]\!]$ be the multiplicity 1, $n$-current associated to the graph of $u$. Recall that for the orientation of a current associated to the graph of a function we use the downward pointing unit normal to the graph. In our case, for the function $u$, we extend this vector to be an $\R^{n+1}$-valued function in all of $\ov\W\times\R$ that is independent of the $x_{n+1}$-variable and we let $\nu$ denote this extension, i.e. for any $(x', x_{n+1})\in\W\times\R$
\begin{equation}\label{nu}
\nu(x', x_{n+1})=\left(\frac{D_1u(x')}{\sqrt{1+|Du(x')|^2}},\dots,\frac{D_nu(x')}{\sqrt{1+|Du(x')|^2}},\frac{-1}{\sqrt{1+|Du(x')|^2}}\right).
\end{equation}
Furthermore, we associate to the vector field $\nu$ an $n$-form $\o$ defined as follows:
\begin{equation}\label{omega}
\omega=\sum_{i=1}^{n+1}(-1)^{i+1}e_i\cdot{\nu} dx_1\wedge\dots\wedge\widehat{dx_i}\wedge\dots\wedge
dx_{n+1}
\end{equation}
where $e_1,e_2,\dots, e_{n+1}$ denote the standard unit vectors in $\R^{n+1}$.

\subsection*{Volume Bounds}

In this paragraph we will show bounds for the mass of the current $T$ and also prove that it has an ``almost minimizing" property (cf. Lemma \ref{minimizing property of Tk}). The main ingredient is  Lemma \ref{dvg-lemma}, which allows us to compare $T$, with other currents that have the same boundary and coincide with $T$ outside $(\ov{\W}\cap B^n_r(0))\times\R$. Recall that $r$ is such that for the initial data of the Dirichlet problem \ref{PMC dirichlet problem with divergence term} we have that $(\W,\graph\phi)\in\B^\a_r$.

\begin{lemma}\label{dvg-lemma}
Assume $u\in C^{1,\a}(\ov{\W})$ is a \emph{(}weak\emph{)} solution of the Dirichlet problem \emph{\ref{PMC dirichlet problem with divergence term}} and let $T=[\![\graph u ]\!]$ be the corresponding multiplicity \emph{1}, $n$-current. Let $R$ be a multiplicity \emph{1}, $(n+1)$-current in $\R^{n+1}$ with $\spt R\subset W\subset\subset (\ov\W\cap B^n_r(0))\times\R$. Then for $S=T-\partial R$ 
\[ (1-d^\a[f]_{\a, B^n_r(0)\times\R})\mass({T})\le(1+d^\a[f]_{\a, B^n_r(0)\times\R})\mass(S)+ \|H\|_{0, B^n_r(0)\times\R} \mass(R)\]
where $d=\diam W$. 
\end{lemma}
\begin{proof}
Note first that if $u$ is smooth then $\o$ (as defined in \ref{omega}) is a smooth $n$-form and hence
\[\begin{split}T(\o)-S(\o)&=\partial R(\o)= R( d\o)= \int_{\spt R}\Theta(x) \dvg \nu(x) dx\\
&=\int_{\spt R}\Theta(x) \left(H(x', u(x'))+\dvg f(x',u(x')) \right)dx'dx_{n+1}\\
&=\int_{\spt R}\Theta(x) \left(H(x', u(x'))+\dvg (f(x',u(x'))-f(x_0) \right)dx'dx_{n+1}\end{split}\]
where $x_0$ is any given point in $W$, $\Theta(x)$ depends on the orientation of $R$, in particular
\[\Theta(x)=<\vec{R}(x),dx_1\wedge\dots\wedge dx_{n+1}>\,\in\{1,-1\}\]
and
recall that for any point $x\in\R^{n+1}$ we use the notation $x=(x', x_{n+1})$.

Hence we get that
\begin{equation*}\label{tso}
T(\o)-S(\o)=\int_{\spt R} \Theta(x) H(x',x_{n+1})dx'dx_{n+1}+\partial R(\o_{f-f_0})
\tag{(1)}
\end{equation*}
where
\[\omega_{f-f_0}=\sum_{i=1}^{n+1}(-1)^{i+1}e_i\cdot\left({f-f(x_0)}\right) dx_1\wedge\dots\wedge\widehat{dx_i}\wedge\dots\wedge
dx_{n+1}\]
which implies the lemma, since $T(\o)=\mass (T)$ and $|\o_{f-f_0}|\le d^\a[f]_{\a, B^n_r(0)\times\R}$ everywhere in $W$.

For the general case, when $u$ is $C^{1,\a}$ it suffices to show that \ref{tso} is still true. For this reason we will approximate $\o$ by smooth $n$-forms.

Let $\zeta\in C^\infty_c(\R^{n+1})$ be such that $\spt\zeta\subset B^{n+1}_1(0)$, $\zeta\ge 0$, $\int_{\R^{n+1}}\zeta(x)dx=1$. For $\s\in (0,1)$ we let $\zeta_\s(x)=\s^{-(n+1)}\zeta(x/\s)$ and consider the $n$-form
\[\o_\s=\zeta_\s*\o=\sum_{i=1}^{n+1}(-1)^{i+1}(\zeta_\s*\nu_i)dx_1\wedge\dots\wedge\widehat{dx_i}\wedge\dots\wedge
dx_{n+1}\]
where $\nu, \o$ are as defined in \ref{nu}, \ref{omega}.

Then $\o_\s$ is a smooth $n$-form and hence
\begin{align*}
T(\o_\s)-S(\o_\s)&=\partial R(\o_\s)=R(d\o_\s)\\
&=\int_{\spt R}\Theta(x)\int_{B^{n+1}_\s(x)}- D_y\left(\zeta_\s(x-y)\right)\cdot\nu(y)dy dx.
\end{align*}
Using equation \ref{MCE*} (the weak form of the prescribed mean curvature equation) we get
\begin{align*}
T(\o_\s)-S(\o_\s)=&-\int_{\spt R}\Theta(x)\int_{B^{n+1}_\s(x)}\zeta_\s(x-y) H(y',u(y'))dydx\\
&-\int_{\spt R}\Theta(x)\int_{B^{n+1}_\s(x)}D_x\zeta_\s(x-y)\cdot \left(f(y',u(y'))-f(x_0)\right)dydx
\end{align*}
where we have used the fact that
\[\int_{B^{n+1}_\s(x)}D_x\zeta_\s(x-y)\cdot f(x_0)dy=0.\]

Hence we have that
\begin{equation*}
T(\o_\s)-S(\o_\s)=-\int_{\spt R}\Theta(x)\int_{B^{n+1}_\s(x)}\zeta_\s(x-y) H(y',u(y'))dydx-\partial R(\zeta_\s*\o_{f-f_0})
\end{equation*}
which by letting $\s\to 0$ implies \ref{tso}.
\end{proof}

\begin{lemma}\label{minimizing property of Tk}
Assume $u\in C^{1,\a}(\ov{\W})$ is a \emph{(}weak\emph{)} solution of the Dirichlet problem \emph{\ref{PMC dirichlet problem with divergence term}} and let $T=[\![\graph u ]\!]$ be the corresponding multiplicity \emph{1}, n-current. Then for any $x_0\in B^n_r(0)\times\R$ and $\r>0$ such that $B^{n+1}_\r(x_0)\subset B^n_r(0)\times\R$ and $\r^\a [f]_{\a,B^n_r(0)\times\R}<1/4$:
\begin{equation*} 
\mass(T\res B^{n+1}_\r(x_0))\le c\left(1+\r\|H\|_{0,B^n_r(0)\times\R}\right)\o_n\r^n
\end{equation*}
and
\begin{equation*}
\mass_W(T)\le\mass_W(S)+c\o_n\r^n\left(\r\|H\|_{0,B^n_r(0)\times\R}+\r^\a[f]_{\a,  B^n_r(0)\times\R}\right)
\end{equation*}
for any $W\subset\subset B^{n+1}_\r(x_0)$ and $S$ an integral $n$-current in $\R^{n+1}$ with $\partial S=\partial T$ and $\spt (T-S)$ a compact subset of $W\cap(\ov{\W}\times\R)$ and where $c$ is an absolute constant.
\end{lemma}

\begin{proof}
Since $\spt (S-T)\subset  B_\r^{n+1}(x_0)\cap\left((B^n_r(0)\cap\ov{\W})\times\R\right)$, by Lemma \ref{dvg-lemma} and using the assumption $\r^\a [f]_{\a,B^n_r(0)\times\R}<1/4$ we have that
\begin{align*}\label{mass inequality}
\left(1-2\r^\a[f]_{\a, B^n_r(0)\times\R}\right)\mass_W(T)\le \mass_W(S)+2\o_n\r^{n+1}\|H\|_{0,B^n_r(0)\times\R}
\tag{(1)}
\end{align*}

Let $U=\{(x',x_{n+1})\in\W\times\R:x_{n+1}\le u(x')\}$, i.e. $U$ is the region under the graph of $u$ and let also $U_\s= U\cap B^{n+1}_\s(x_0)$. By Sard's theorem, for almost all $\s>0$, $\mass(T\res\partial B^{n+1}_\s(x_0))=0$. For such $\s\le\r$, let $S=T\res B^{n+1}_\s(x_0)-\partial U_\s$. Then 
$\spt S\subset \partial B^{n+1}_\s(x_0)\cup\left((\partial\W\times\R)\cap B_\s^{n+1}(x_0)\right)$ and so
\[\mass(S)\le 4\o_n\s^n+\o_n \s^n(1+\k)\le 6\o_n\s^n\]
where $\k=\k_{(\W,\Phi)}$ is as in Definition \ref{ar regular class}.
Using this in inequality \ref{mass inequality} we have that
\[\mass(T\res B^{n+1}_\s(x_0))\le 12(1+\s\|H\|_{0,B^n_r(0)\times\R})\o_n\s^n\]
which gives the first assertion of the lemma.

Taking $\s\in (0,\r]$ such that $W\subset B^{n+1}_\s(x_0)$ we have that
\[\mass_W(T)\le 12(1+\r\|H\|_{0,B^n_r(0)\times\R})\o_n\r^n\]
and using this estimate back in the inequality \ref{mass inequality} we get
\[\begin{split}\mass_W(T)\le \mass_W(S)&+24\o_n\r^n(1+\r\|H\|_{0,B^n_r(0)\times\R})\r^\a[f]_{\a, B^n_r(0)\times\R}\\
&+2\o_n\r^{n+1}\|H\|_{0,B^n_r(0)\times\R}\end{split}\]
which implies the second assertion of the lemma.
\end{proof}

\begin{definition}\label{sclose}
Let $M$ be an $n$-dimensional manifold in $\R^{n+1}$, $x\in M$ and $P$ an $n$-dimensional linear space passing through $x$. We say that $M$ is $\s$-close to $P$ in $B^{n+1}_\r(x)$ if
\[M\cap B^{n+1}_\r(x)\subset q(Q_{\r,\s})\]
for some orthogonal transformation $q$ of $\R^{n+1}$ such that $q(0)=x$, $q(\{0\}\times\R^n)=P$ and
where 
\[Q_{\r,\s}=[-\s\r,\s\r]\times B_\r^n(0).\]

\end{definition}

\begin{lemma}\label{connectivity lemma}
Let $u\in C^{1,\a}(\ov{\W})$ be a \emph{(}weak\emph{)} solution of the Dirichlet problem \emph{\ref{PMC dirichlet problem with divergence term}}. Let $x_0\in\graph u\cap (B^n_{r/2}(0)\times\R)$ and $\r\in(0,r/2]$ be such that $B^{n+1}_\r(x_0)\subset B^n_{r/2}(0)\times\R$, $B^{n+1}_\r(x_0)\cap\graph\phi=\emptyset$ and $\r^\a[f]_{\a,B^n_{r}(0)\times\R}<1/4$. Then
\[\H^n(\graph u\cap q(Q_{\r,\s}))\le (1+3\r^\a [f]_{\a,B^n_r(0)\times\R})\o_n\r^n+c\s\o_n\r^{n}(n+\r\|H\|_{0,B^n_r(0)\times\R})\]
for any $\s\in(0,\r)$ and any orthogonal transformation of $\R^{n+1}$, $q$, such that $q(0)=x_0$ and where $c$ is an absolute constant.

\end{lemma}
\begin{proof}

Let $q$ be an orthogonal transformation of $\R^{n+1}$ and let $Q^\pm$ be the regions in $q(Q_{\r,\s})$ that lie above and below the graph of $u$, i.e.
\[Q^+=\{x=(x',x_{n+1})\in q(Q_{\r,\s}):x_{n+1}> u(x)\}\]
\[Q^-=\{x=(x',x_{n+1})\in q(Q_{\r,\s}):x_{n+1}< u(x)\}.\]

Notice that for one of the $\partial Q^\pm$, say $\partial Q^+$, we know that
\[ \left|\partial Q^+\cap \left(q\left(\{\s\r\}\times B^n_\r(0)\right)\cup q\left(\{-\s\r\}\times B^n_\r(0)\right)\right)\right| \le\o_n\r^n.\] 
Then the lemma is a direct consequence of Lemma \ref{dvg-lemma}, applied with $Q^+$ in place of $R$.
\end{proof}

\begin{rmk}\label{rmk of connectivity lemma}
If in addition to the hypotheses of \emph{Lemma \ref{connectivity lemma}} we have that $\graph u\cap B^{n+1}_\r(x_0)$ is $\s$-close to a plane $($in the sense of \emph{Definition \ref{sclose}}$)$ then \emph{Lemma \ref{connectivity lemma}} eventually gives
\[\H^n(\graph u\cap B^{n+1}_\r(x_0))\le (1+3\r^\a [f]_{\a,B^n_r(0)\times\R})\o_n\r^n+c\s\o_n\r^{n}(n+\r\|H\|_{0,B^n_r(0)\times\R}).\]
\end{rmk}

\begin{lemma}\label{small enclosed volume}
Let $u\in C^{1,\a}(\ov{\W})$ be a \emph{(}weak\emph{)} solution of the Dirichlet problem \emph{\ref{PMC dirichlet problem with divergence term}}. Let $x_0\in\graph u\cap (B^n_r(0)\times\R)$ and $\r>0$ be such that $B^{n+1}_\r(x_0)\subset B^n_r(0)\times\R$, $B^{n+1}_\r(x_0)\cap\graph \phi=\emptyset$ and $\r^\a[f]_{\a, B^n_r(0)\times\R}\le 1/2$. Then
\begin{equation*}\label{small enclosed volume eqn}
\H^{n+1}(U^\pm\cap B^{n+1}_\r(x_0))\ge c\r^{n+1}
\tag{(1)}
\end{equation*}
where $U^\pm$ are the regions of $\W\times\R$ that lie above $(U^+)$ and below $(U^-)$ the graph of $u$  and
\begin{equation*}\label{lower bound for area}
\H^{n}(\graph u\cap B^{n+1}_\r(x_0))\ge c\r^{n}.
\tag{(2)}
\end{equation*}
The constant $c$ in both inequalities depends on $\|H\|_{0,B^n_r(0)\times\R}$ and $n$.
\end{lemma}
\begin{proof}
We first give the proof of \ref{small enclosed volume eqn} for $U^-$; the argument for $U^+$ is similar.

 Let
 \begin{equation*}\label{ur}
U_\r=U^-\cap B^{n+1}_\r(x_0)\,,\,G_\r=\graph u\cap B^{n+1}_\r(x_0).
\tag{(3)}
\end{equation*}
By Lemma \ref{dvg-lemma} (with $ U_\r$ in place of $R$) we get that
\begin{equation*}\label{gr}
\H^n(G_\r)\le 3\frac{d}{d\r}\H^{n+1}(U_\r)+2 \|H\|_{0,B^n_r(0)\times\R}\H^{n+1}(U_\r).
\tag{(4)}
\end{equation*}
Since
\[\H^n(\partial U_\r)\le \H^{n}(G_\r)+\frac{d}{d\r}\H^{n+1}(U_\r)\]
the isoperimetric inequality for $U_\r$ implies that
\begin{equation*}\label{iso}
\H^{n+1}(U_\r)^\frac {n}{n+1}\le c(n)\H^n(\partial U_\r)\le c(n)\left(\H^{n}(G_\r)+\frac{d}{d\r}\H^{n+1}(U_\r)\right).
\tag{(5)}
\end{equation*}
Using the estimate \ref{gr} in \ref{iso} we get
\[\H^{n+1}(U_\r)^\frac {n}{n+1}\le c\frac{d}{d\r}\H^{n+1}(U_\r)\]
because we can assume that $2c(n)\|H\|_{0,B^n_r(0)\times\R}\H^{n+1}(U_\r)\le\frac12 \H^{n+1}(U_\r)^\frac{n}{n+1}$, where $c(n)$ is the constant from the isoperimetric inequality, since otherwise the lemma is trivially true. 

Hence
\[\frac{d}{d\r}\left(\H^{n+1}(U_\r)^{\frac {1}{n+1}}\right)\ge c\]
and after integrating
\[\H^{n+1}(U_\r)\ge c\r^{n+1}.\]

For proving \ref{lower bound for area} of the lemma we let $U_\r$, $G_\r$ be as defined in \ref{ur} above. By inequality \ref{small enclosed volume eqn} we know that
\[\H^{n+1}(U_\r)\ge c\r^{n+1}.\]

Let $v$ be a unit vector in $\R^{n+1}$ such that 
\begin{equation*}\label{vprop}
v\cdot (0,\dots,0,1)>0.
\tag{(6)}
\end{equation*}
For such a vector $v$ we define
$P_v$ to be the $n$-dimensional affine subspace of $\R^{n+1}$, passing through $x_0$ and normal to $v$, i.e.
\[P_v=\{x\in\R^{n+1}:(x-x_0)\cdot v=0\}\]
and 
$U^+_{\r,v}$ to be the part of $U_\r$ that lies above $P_v$, i.e.
\[U^+_{\r,v}=\{x\in U_\r: (x-x_0)\cdot v>0\}.\]

We claim that it is enough to prove that for some vector $v$, satisfying \ref{vprop}, we have that
\begin{equation*}\label{above P small}
\H^{n+1}(U^+_{\r,v})\le\frac14\H^{n+1}(U_{\r}).\tag{(7)}
\end{equation*}
To see this assume that \ref{above P small} is true for some $v$ and let $G^\pm_{\r,v}$ be the parts of $G_\r$ that lie above ($G^+_{\r,v}$) and below ($G^-_{\r,v}$) the affine subspace $P_v$. Let also
\[\wt{G}_\r=G^-_{\r,v}\cup \refl_{P_v}(G^+_{\r,v})\cup\left(G_\r\cap P_v\right)\]
where $\refl_{P_v}$ denotes the reflection along $P_v$ and let $\wt{U}_\r$ be the region of $B^{n+1}_\r(x_0)$ that lies below $\wt{G}_\r$. Then we have that
\[\H^n(\wt{G}_\r)=\H^n(G_\r)\]
and
\[\H^{n+1}(\wt{U}_\r)\ge\H^{n+1}(U_{\r})-2\H^{n+1}(U^+_{\r,v})\ge\frac12\H^{n+1}(U_{\r}).\]
Furthermore, since $\wt{U}_\r$ lies below $P_v$, we have that
\[\proj_{P_v}(\wt{G}_\r)=\proj_{P_v}(\wt{U}_\r)\]
where $\proj_{P_v}$ denotes the projection onto the affine subspace $P_v$. Finally since
\[\wt{U}_\r\subset\{x-t v:x\in\proj_{P_v}(\wt{U}_\r), 0\le t\le\r\}\]
we have that
\[\r\H^n(\proj_{P_v}(\wt{U}_\r))\ge\H^{n+1}(\wt{U}_\r)\]
and hence
\[\begin{split}\H^n(G_\r)&=\H^n(\wt{G}_\r)\ge\H^n(\proj_{P_v}(\wt{G}_\r))=\H^n(\proj_{P_v}(\wt{U}_\r))\\
&\ge\r^{-1}\H^{n+1}(\wt{U}_\r)\ge\r^{-1}\frac12\H^{n+1}(U_{\r})\ge c\r^n\end{split}.\]

We now need to show that for some vector $v$ that satisfies \ref{vprop}, inequality \ref{above P small} is true.

For any $t\in (-1,1)$ let $v_t=(t,0,\dots,0,\sqrt{1-t^2})$,
\[P_t=\{x\in\R^{n+1}:(x-x_0)\cdot v_t=0\}\]
and
\[U_{\r,t}^+=\{x\in U_\r:(x-x_0)\cdot v_t>0\}\,\,,\,\,U_{\r,t}^-=\{x\in U_\r:(x-x_0)\cdot v_t<0\}.\]
We claim that for some $t$, $\H^{n+1}(U_{\r,t}^+)\le\frac14 \H^{n+1}(U_\r)$.

Assume it is not true. Then for all $t$
\[\H^{n+1}(U_{\r,t}^-)\le\frac14 \H^{n+1}(U_\r)\]
and hence for any $\e>0$
\[\H^{n+1}(U_{\r,1-\e}^+\cap U_{\r,-1+\e}^+)\ge\frac 12\H^{n+1}(U_\r)\]
which is impossible, since
\[\H^{n+1}(U_{\r,1-\e}^+\cap U_{\r,-1+\e}^+)\stackrel{\e\to0}{\longrightarrow}0.\]

\end{proof}

\subsection*{The solutions to \ref{PMC dirichlet problem with divergence term} are uniformly close to planes near the boundary cylinder}

In this paragraph we want to show that given $\e>0$, there exists $\r>0$ depending only on $\e$, $r$, $\|H\|_{0,B^n_r(0)\times\R}$ and $[f]_{\a, B^n_r(0)\times\R}$, such that the graph of $u$ is $\e$-close to some $n$-dimensional linear space in all balls of radius less than $\r$ that intersect the boundary cylinder (cf. Theorem \ref{close-to-plane thm}). The main ingredient is the following lemma:

\begin{lemma}\label{graph current convergence to a plane}
Let $(\W_k,\graph\phi_k)\in \B^\a_{r_k}$ be a sequence such that $r_k\to\infty$, $\k_{(\W_k,\Phi_k)}\to 0$ and $\partial\W_k\cap B_r^n(0)\ne\emptyset$ for some $r\in(0,\infty)$, where $\Phi_k=\graph\phi_k$. Assume that $u_k\in C^{1,\a}(\ov{\W}_k)$ is a \emph{(}weak\emph{)} solution of the corresponding Dirichlet problem \emph{\ref{PMC dirichlet problem with divergence term}}
with $[f_k]_{\a,B^{n}_{r_k}(0)\times\R}\to 0$ and $\|H_k\|_{0, B^{n}_{r_k}(0)\times\R}\to 0$ and let $T_k=[\![\graph u_k ]\!]$.

Then, after passing to a subsequence,
\begin{equation*}\label{rescaled current convergence equation}
T_k\res (B^{n}_{r_k}(0)\times\R)\to T
\tag{(1)}
\end{equation*}
in the weak sense of currents, but also with the corresponding measures converging 
$\mu_{T_k}\to\mu_T$ 
as Radon measures and where for the limit $T$ either
\begin{enumerate}
\item[{(i)}] $\partial T=0$ and $\spt T$ is a vertical hyperplane
\end{enumerate}
or
\begin{enumerate}
\item[{(ii)}] $\partial T\ne 0$ and $\spt T$ is an $n$-dimensional halfspace.
\end{enumerate}

Furthermore for the convergence in \emph{\ref{rescaled current convergence equation}} we have that for any $\e>0$ and $W$ a compact subset of $\R^{n+1}$ such that $W\cap\spt T\ne\emptyset$, there exists $k_0$ such that for all $k\ge k_0$
\begin{equation*}\label{Hausdorff close}
\spt T_k\cap W\subset \e\text{-neighborhood of }\spt T.
\tag{(2)}
\end{equation*}

\end{lemma}

\begin{proof}

We note that given any $\r>0$ and for all $k$ large enough we have that $\r^\a[f_k]_{\a,B^{n}_{r_k}(0)\times\R}<1/4$ and therefore we can apply Lemma \ref{minimizing property of Tk}. This implies that  the currents $T_k\res (B^{n}_{r_k}(0)\times\R)$ have locally uniformly bounded masses. Hence we can apply the Federer-Fleming compactness theorem \cite[Theorem 32.2]{LSgmt} which implies that after passing to a subsequence
\[{T}_k\res (B^{n}_{r_k}(0)\times\R)\to T\]
in the weak sense of currents in $\R^{n+1}$, where $T$ is an integral $n$-current. Furthermore, by Remark \ref{AA remark}, $T$ has support in an $(n+1)$-dimensional closed halfspace. Without loss of generality, we can assume that this halfspace is equal to $\ov{\R}_+\times\R^n=\{x\in\R^{n+1}:x_1\ge 0\}$. According to the Federer-Fleming compactness theorem we also have that $\partial{T}_k=[\![\Phi^{k}]\!]\to \partial T$ and thus (using Remark \ref{AA remark} again) either $\partial T=0$ or $\partial T=[\![\Phi]\!]$, where $\Phi$ is an $(n-1)$-dimensional affine subspace of $\{0\}\times\R^n$.

We claim that $T$ is area minimizing. In view of Lemma \ref{area minimizing in halfspace} it suffices to prove that it is area minimizing in the closed halfspace $\ov{\R}_+\times\R^n$. Note that although the currents ${T}_k$ are not area minimizing, they do satisfy a minimizing property (cf. Lemma \ref{minimizing property of Tk}), which enables us to argue as in the case when $T$ is the limit of area minimizing currents (cf. \cite[Theorem 34.5]{LSgmt}), as follows:

Since ${T}_k\res (B^{n}_{r_k}(0)\times\R)\to T$ in the weak sense of currents, we know that the convergence is also with respect to the flat-metric (cf. \cite[Theorem 31.2]{LSgmt}), i.e. there exist integral $(n+1)$-currents $R_k$ and integral  $n$-currents $P_k$ such that $T-{T}_k\res (B^{n}_{r_k}(0)\times\R)=\partial R_k+P_k$ and for any compact subset $W$ of $R^{n+1}$
\[\mass_W(R_k)+\mass_W(P_k)\to0.\]

Let $S$ be an integral $n$-current such that $\partial S=0$ and $\spt S\subset W\cap (\ov{\R}_+\times\R^n)$, where $W$ is a compact subset of $\R^{n+1}$. Let $W_\e=\{x\in\R^{n+1}:\dist(x,W)<\e\}$. We can choose $\e\in(0,1)$ so that, after passing to a subsequence, we have that for all $k$:
\begin{equation*}\label{zero masses}
\mass({T}_k\res\partial W_\e)=0\,,\, \mass(T\res\partial W_\e)=0
\tag{(3)}
\end{equation*}
and
\begin{equation*}\label{zero masses 2}
\partial(R_k\res W_\e)=(\partial R_k)\res W_\e+L_k
\tag{(4)}
\end{equation*}
where $L_k$ is an integral $n$-current such that $\spt L_k\subset\partial W_\e$ and $\mass(L_k)\to 0$. 
Then
\[T\res W_\e-{T}_k\res W_\e=\partial \wt{R}_k+\wt{P}_k\]
where $\wt{R}_k=R_k\res W_\e$ and $\wt{P}_k= P_k\res W_\e-L_k$, and
\begin{equation*}\label{TS est}
\mass_{W_\e}(T+S)=\mass_{W_\e}({T}_k+S+\partial \wt{R}_k+\wt{P}_k)\ge\mass_{W_\e}(T_k+S_k)-\mass_{W_\e}(\wt{P}_k)
\tag{(5)}
\end{equation*}
where $S_k=S+\partial \wt{R}_k$. For $S_k$ we have that $\spt S_k\subset \ov{W}_\e$ and $\partial S_k=0$. 

Let $R>0$ be such that $\ov{W}_\e\subset B^{n+1}_R(0)$. For $k$ big enough, so that $r_k>R$, let $\ell_k$ be a lipschitz retraction of $(\ov\R_+\times\R^n)\cap B^{n+1}_R(0)$ to $(\ov{\W}_k\times\R)\cap B^{n+1}_R(0)$ such that
\[\frac{|\ell_k(x)-\ell_k(y)|}{|x-y|}\le\k_{(\W_k,\Phi_k)}\,,\,\forall x,y\in \{0\}\times\R^{n}.\]
Then the current $\ell_{k\#}S_k$ has support in $\ov{\W}_k\times\R$ and no boundary, hence using the minimizing property of $T_k$ (Lemma \ref{minimizing property of Tk}) we have that
\[\begin{split}\mass_{{W}_{\e}}({T}_k)&\le \mass_{{W}_{\e}}(T_k+\ell{_{k\#}}S_k)+c\left(R\|H\|_{0,B^n_R(0)\times\R}+R^\a[f]_{\a,  B^n_R(0)\times\R}\right)
\o_nR^n\\
&\le |J\ell_k|\mass_{{W}_{\e'}}(T_k+S_k)+c\left(R\|H\|_{0,B^n_R(0)\times\R}+R^\a[f]_{\a,  B^n_R(0)\times\R}\right)
\o_nR^n\end{split}\]
for any $\e'>\e$ such that $W_{\e'}\subset B^{n+1}_R(0)$ and where $J\ell_k$ denotes the Jacobian of $\ell_k$ and thus $|J\ell_k|<1+c\k_{(\W_k,\Phi_k)}$. Letting $\e'\downarrow \e$ we get
\begin{equation*}\label{epsilon}
\begin{split}\mass_{{W}_{\e}}({T}_k)\le & |J\ell_k|\left(\mass_{{W}_{\e}}(T_k+S_k)+\mass(L_k)\right)\\
&+c\left(R\|H\|_{0,B^n_R(0)\times\R}+R^\a[f]_{\a,  B^n_R(0)\times\R}\right)
\o_nR^n\end{split}
\tag{(6)}
\end{equation*}
where we have used \ref{zero masses} and \ref{zero masses 2}.
Hence, using \ref{epsilon} to estimate $\mass_{{W}_{\e}}(T_k+S_k)$ in \ref{TS est}, we get
\[\begin{split}\mass_{W_\e}(T+S)\ge (1&-c\k_{(\W_k,\Phi_k)})\mass_{W_\e}(T_k)-\mass_{W_\e}(\wt{P}_k)-\mass(L_k)\\
&-c\left(R\|H\|_{0,B^n_R(0)\times\R}+R^\a[f]_{\a,  B^n_R(0)\times\R}\right)\o_nR^n\end{split}\]
where $c$ depends only on $n$. Letting $k\to\infty$ and using the lower semicontinuity of the mass and the fact that $\mass_{W_\e}(\wt{P}_k)\to 0$, $\mass(L_k)\to 0$ we get
\[\mass_{{W}_\e}(T)\le\mass_{{W}_\e}(T+S)\]
which implies that
\[\mass_W(T)\le\mass_W(T+S)\]
since $S=0$ outside $W$, and hence $T$ is area minimizing.

We claim now that $\mu_{T_k}\to\mu_T$ as Radon measures.

Using $S=0$ in the above argument we have that
\[\begin{split}\mass_{W_\e}(T)\ge (1&-c\k_{(\W_k,\Phi_k)})\mass_{W_\e}(T_k)-\mass_{W_\e}(\wt{P}_k)\\
&-c\left(R\|H\|_{0,B^n_R(0)\times\R}+R^\a[f]_{\a,  B^n_R(0)\times\R}\right)\o_nR^n\end{split}\]
and letting $k\to 0$
\[\mass_{W_\e}(T)\ge\limsup_k\mass_{W_\e}({T}_k).\]
Since $W\subset W_\e$
\[\limsup\mu_{T_k}(W)\le\mass_{W_\e}(T)=\mu_T(W_\e)\]
and because we can repeat the argument for $\e\downarrow 0$ we have that
\[\mu_T(W)\ge \limsup_k\mu_{{T}_k}(W)\]
which along with the lower semicontinuity of Radon measures implies the measure convergence.

Next we will show that $T$ is either an $n$-dimensional halfspace or a vertical hyperplane.

Assume first that we are in case (i) $\partial T=0$. This is the case when for any $W\subset\subset\R^{n+1}$, $W\cap \Phi_k=\emptyset$  and hence $\partial T_k\res W=0$, for all $k$ large enough (cf. Remark \ref{AA remark}).

Using the uniform area ratio bounds, Lemma \ref{minimizing property of Tk} and the interior monotonicity formula \cite{Afirstvarofvarifold} we have 
\[1\le \o_n^{-1}r^{-n}\mu_T(B^{n+1}_r(x))=\o_n^{-1}r^{-n}\lim_k\mu_{{T}_k}(B^{n+1}_r(x))\le c\]
for all $x\in\spt T$ and any $r>0$, where $c$ is an absolute constant.

Hence for a sequence $\{\Lambda_i\}\uparrow\infty$ we can apply the Federer-Fleming compactness theorem to the sequence $T_{x,\Lambda_i}=\eta_{x,\Lambda_i\#}T$, where for $x\in \R^{n+1}$ and $\l\in\R$, $\eta_{x,\l}:\R^{n+1}\to\R^{n+1}$ is defined by $\eta_{x,\l}(y)=\l^{-1}(y-x)$. So, after passing to a subsequence,
\[T_{x,\Lambda_i}\to C\]
in the weak sense of currents, where $C$ is an integral $n$-current. Since $T_{x,\Lambda_i}$ are area minimizing, $C$ is an area minimizing cone and $\mu_{T_{x,\Lambda_i}}\to \mu_C$ as radon measures. Furthermore, since $\spt T\subset\ov{\R}_+\times\R^n$ we have that $\spt T_{x,\Lambda_i}\subset\{y\in\R^{n+1}:y_1\ge-\Lambda_i^{-1}x_1\}$, where $x_1, y_1$ denote the first coordinates of $x$ and $y$ respectively, and hence $\spt C\subset \ov{R}_+\times\R^n$. This implies \cite[Theorems 36.5, 26.27]{LSgmt} that 
\[C=m[\![\{0\}\times\R^n]\!]\]
for some integer $m\ge 1$. We claim that in fact $m=1$.

For $\s\in (0,1)$, let $Q_{1,\s}=[-\s,\s]\times B^{n}_1(0)$. Then $\mu_C(Q_{1,\s})=m\o_n$. By the measure convergence $\mu_{T_{x,\Lambda_i}}\to \mu_C$ and $\mu_{T_k }\to \mu_T$, we have that for any $\d>0$ there exists some $\Lambda>0$ and $k_0$ such that for all $k\ge k_0$ 
\[m-\d\le\frac{1}{\Lambda^n\o_n}\mu_{T_k}(x+\Lambda Q_{1,\s}).\]
Using Lemma \ref{connectivity lemma}, the RHS of the above inequality is less than $1+\Lambda^\a[f_k]_{\a, B^n_{r_k}(0)\times\R}+c\s(1+\Lambda\|H_k\|_0)$ and hence taking $\s$ small enough we conclude that $m$ has to be 1.

Hence we get that
\[\o_n^{-1}r^{-n}\mu_T(B^{n+1}_r(x))=1\,,\forall x\in\spt T\text {  and  }r>0\]
which implies that $T$ itself is a hyperplane of multiplicity 1 and since $\spt T\subset\ov{R}_+\times\R^n$ is has to be a vertical hyperplane.

Assume now that we are in case (ii) $\partial T=[\![\Phi]\!]$. In this case $\Phi_k\toap\Phi$ for all $\a'<\a$ and hence $\Phi$ is an $(n-1)$- dimensional linear subspace of $\{0\}\times\R^n$ (cf. Remark \ref{AA remark}). Without loss of generality we can assume that $\Phi=\{0\}\times\R^{n-1}\times\{0\}$.

Using the uniform area ratio bounds, Lemma \ref{minimizing property of Tk} and the boundary monotonicity formula \cite{Abdryregularity} we get
\[\frac12\le \o_n^{-1}r^{-n}\mu_T(B^{n+1}_r(0))=\o_n^{-1}r^{-n}\lim_k\mu_{{T}_k}(B^{n+1}_r(0))\le c\]
for any $r>0$ and where $c$ is an absolute constant.

Hence for a sequence $\{\Lambda_i\}\uparrow\infty$ we can apply the Federer-Fleming compactness theorem to the sequence $T_{\Lambda_i}=\eta_{0,\Lambda_k\#}T$ to conclude (as in case (i)) that after passing to a subsequence
\[T_{\Lambda_i}\to C\]
where $C$ is an area minimizing cone with $\spt C\subset\ov{\R}_+\times\R^n$, $\partial C=\Phi=\{0\}\times\R^{n-1}\times\{0\}$ and also $\mu_{T_{\Lambda_i}}\to \mu_C$ as Radon measures.

Hence we can apply Lemma \ref{cone as a union of planes} and in particular Corollaries \ref{corollary cone union of planes}, \ref{corollary cone in plane} to $C$ to conclude that $C$ is either an $n$-dimensional halfspace or
\[C=mP_1+(m-1)P_2\]
for some integer $m\ge 1$, where $P_1, P_2$ denote the $n$-dimensional halfspaces $\{0\}\times\R^{n-1}\times\R_\pm$. 

We claim that in the latter case $m=1$ and hence $C$ is a halfspace in either case.
To see this, take $x\in\spt C$ such that for $Q_{1,\s}(x)=x+[-\s,\s]\times B_1^n(0)$ we have that $Q_{1,\s}(x)\cap\Phi=\emptyset$ and $\mu_C(Q_{1,\s}(x))=m\o_n$. We can argue now as in case (i) and using the measure convergence $\mu_{T_{\Lambda_i}}\to \mu_C$, $\mu_{T_k}\to\mu_T$ and Lemma \ref{connectivity lemma}  we have that for any $\d>0$ there exists some $\Lambda>0$ and $k_0$ such that for all $k\ge k_0$ 
\[m-\d\le\frac{1}{\Lambda^n\o_n}\mu_{T_k}(x+\Lambda Q_{1,\s}(x))\le 1+ 3\Lambda^\a [f_k]_{\a, B^n_{r_k}(0)\times\R}+c\s(1+\Lambda\|H_k\|_0)\]
and hence taking $\s$ small enough we conclude that $m$ has to be 1.

Hence $C$ is a halfspace and therefore for any $r>0$
\[\o_n^{-1}r^{-n}\mu_T(B^{n+1}_r({0}))=\frac 12\]
so that $T$ is an area minimizing cone with vertex $0$. Hence we can apply Lemma \ref{cone as a union of planes} and Corollaries \ref{corollary cone union of planes}, \ref{corollary cone in plane} to $T$, which along with the fact that the density at $0$ is $1/2$ imply that $T$ is a an $n$-dimensional halfspace.

We finally have to prove statement \ref{Hausdorff close} of the Theorem.

Assume that for some $W\subset\subset\R^{n+1}$ such that $W\cap\spt T\ne\emptyset$ and $\e>0$,  statement \ref{Hausdorff close} of the Theorem is not true. Hence, after passing to a subsequence, we have that for every $k$ there exists $x_k\in\spt T_k\cap W$ such that 
\[B^{n+1}_{\e}(x_k)\cap\spt T=\emptyset.\]
Since either $\Phi_k\cap W=\emptyset$ for $k$ big enough or $\Phi_k\cap (B^{n}_{r_k}(0)\times\R)\toap \Phi$ for all $\a'<\a$, we have (after passing to a further subsequence if necessary) that $B^{n+1}_{\e/2}(x_k)\cap \Phi_k=\emptyset$. Hence, for $k$ big enough so that $B^{n+1}_{\e/2}(x_k)\subset B^n_{r_k}(0)\times\R$ and $(\e/2)^\a[f_k]_{\a,B^n_{r_k}\times\R}\le1/2$, we can apply  \ref{lower bound for area} of Lemma \ref{small enclosed volume} with $x_0=x_k$ to conclude that for any $\r\le\e/2$
\[\mu_{T_k}(B^{n+1}_{\r}(x_k))\ge c \r^n\]
where $c$ depends only on $n$.

Since $x_k\in W\subset\subset\R^{n+1}$ and $B^{n+1}_{\e}(x_k)\cap\spt T=\emptyset$ for all $k$, we have that, after passing to a subsequence, $x_k\to x_0$ for some $x_0\in W$, such that $B_{\e/2}^{n+1}(x_0)\cap \spt T=\emptyset$.  Then for $k$ large enough we have that
\[ B_{\e/4}^{n+1}(x_k)\subset B_{\e/2}^{n+1}(x_0)\Rightarrow \mu_{T_k}(B_{\e/2}^{n+1}(x_0))\ge \mu_{T_k}(B_{\e/4}^{n+1}(x_k))\ge c(\e/4)^n.\]
By the mass convergence $\mu_{T_k}\to\mu_T$, this implies that
\[\mu_T(B_{\e/2}^{n+1}(x_0))>0\]
which contradicts the fact that $B_{\e/2}^{n+1}(x_0)\cap\spt T=\emptyset$.
\end{proof} 

\begin{theorem}\label{close-to-plane thm}
Let $(\W,\Phi)\in\B^\a_{r}$, with $\Phi$ given by the graph of a function; $\Phi=\graph\phi$ and let also $u\in C^{1,\a}(\ov{\W})$ be a \emph{(}weak\emph{)} solution of the Dirichlet problem \emph{\ref{PMC dirichlet problem with divergence term}} with $H$, $f$ satisfying
\begin{equation*}\label{bdrydatacon}
\|H\|_{0,B^n_r(0)\times\R}\le K \,\,,\,\,[f]_{\a, B^n_r(0)\times\R}\le K.\tag{(1)}
\end{equation*}
for some $K>0$.

Then $\forall\,\e>0$, there exists $\r=\r(r,\e,K)<r$ such that the following holds:\\
For any $x\in (B^n_{r/2}(0)\times\R)\cap\graph u$ and $\l\in(0,\r]$ such that
$\dist(x,\partial\W\times\R)<\l$, 

\begin{equation*}\label{Pisclose}
\l ^{-1}\sup_{\graph u\cap B^{n+1}_{\l }(x)}\dist(y-x,P)<\e
\tag{(2)}
\end{equation*}
for some $n$-dimensional linear subspace $P=P(x,\l)$.
\begin{equation*}\label{massclose}
\o_n^{-1}\l^{-n}|\graph u\cap B^{n+1}_\l(x)|\le1+\e.
\tag{(3)}
\end{equation*}

In particular if $x\in\Phi$ then inequality \emph{\ref{Pisclose}} holds with an $n$-dimensional halfspace $P_+=P_+(x, \l)$ in place of $P$, such that $0\in\partial P_+$ and
\[\l^{-1}\sup_{\Phi\cap B^{n+1}_{\l }(x)}\dist(y-x,\partial P_+)<\e\]
and inequality \emph{\ref{massclose}} holds with the \emph{RHS} replaced by $1/2+\e$.
\end{theorem}

\begin{proof}
Assume that the theorem is not true. Then for some $\e>0$, there exist a sequence of boundary data $(\W_i,\Phi_i)\in \B^\a_{r}$ and corresponding Dirichlet problems (as in \ref{PMC dirichlet problem with divergence term}) with $H_i, f_i$ satisfying \ref{bdrydatacon} such that the following holds: there exists a sequence $\l_i\downarrow 0$ and $x_i\in (B^n_{r/2}(0)\times\R)\cap\graph u_i$, where $u_i\in C^{1,\a}(\W_i)$ are weak solutions of the corresponding problems, with 
\[\dist(x_i,\partial\W_i\times\R)<\l_i\]
 but such that at least one of the assertions \ref{Pisclose}, \ref{massclose} with $x=x_i$ and $\l=\l_i$ fails.

Let $\wt{\W}_i=\eta_{x_i,\l_i}(\W_i)$ and $\wt{\Phi}_i=\eta_{x_i,\l_i}(\Phi_i)$, where $\eta_{x,\l}(y)=\l^{-1}(y-x)$. Then (after a vertical translation so that $\wt{\W}_i\subset\R^n\times\{0\}$)
\[(\wt{\W}_i,\wt{\Phi}_i)\in \B^\a_{ r/(2\l_i)}\]
and  $\k_{(\wt{\W}_i,\wt{\Phi}_i)}\le\l_i^\a\k_{(\W_i,\Phi_i)}$. Also $\wt{\Phi}_i=\graph\wt{\phi}_i$, where $\wt{\phi}_i\in C^{1,\a}(\partial\wt\W_i)$ is defined by $\wt{\phi}_i(x')=\eta_{x_i,\l_i}(\phi_i(\l_ix'+x_i))$.
Furthermore for $\wt{T}_i=\eta_{x_i,\l_i\#}(T_i)$, where $T_i=[\![\graph u_i]\!]$, we have that $\wt{T}_i=[\![\graph\wt{u}_i]\!]$, where $\wt{u}_i\in C^{1,\a}(\wt{\W}_i)$ is defined by $\wt{u}_i(x')=\eta_{x_i,\l_i}(u_i(\l_ix'+x_i))$ and is therefore a solution to the Dirichlet problem
\begin{equation*}\label{approx tilde prob}
\begin{split}
\sum_{j=1}^nD_j\left(\frac{D_j\wt{u}_i}{\sqrt{1+|D\wt{u}_i|^2}}\right)&=\wt{H}_i +\sum_{j=1}^nD_j\wt{f}^j_i \text{  in  } \wt{\W}_i\\
\wt{u}_i&=\wt{\phi}_i \text{  on  }\partial\wt{\W}_i
\end{split}
\end{equation*}
with
\[\wt{H}_i(x)=\l_i H_i(x_i+\l_ix)\Rightarrow \|\wt{H}_i\|_{0, B^n_{r/(2\l_i)}\times\R}\le\l_i\|H_i\|_{0, B^n_r(0)\times\R}\stackrel{i\to\infty}{\longrightarrow}  0\]
\[\wt{f}_i(x)=f_i(x_i+\l_ix)\Rightarrow[\wt{f}_i]_{\a,B_{r/(2\l_i)}^{n}(0)\times\R}\le\l_i^\a[f_i]_{\a,B_{r}^{n}(0)\times\R}\stackrel{i\to\infty}{\longrightarrow} 0.\]
Hence we can apply Lemma \ref{graph current convergence to a plane} to the sequence $\wt{T}_i$ which implies that
\begin{equation*}\label{wtTi to T}
\wt{T}_i\res(B^n_{r/(2\l_i)}(0)\times\R)\to T
\end{equation*}
in the weak sense of currents, but also $\mu_{\wt{T}_i}\to\mu_T$ as Radon measures and where for the limit $T$ the following holds:
\begin{enumerate}
\item[(i)] If $\liminf{\l}_i^{-1}\dist({x}_i,\Phi_i)=\infty$ then $\partial T=0$ and $\spt T$ is a vertical hyperplane.
\item[(ii)] If $\liminf{\l}_i^{-1}\dist({x}_i,\Phi_i)<\infty$ then $\partial T\ne 0$, $\spt T$ is an $n$-dimensional halfspace and $\partial T=[\![\Phi]\!]$ with $\Phi$ being determined by $\wt\Phi_i$ as follows:\\ $\wt{\Phi}_i\cap (B_{r/(2\l_i)}^n(0)\times\R)\toap \Phi$, for all $\a'<\a$.
\end{enumerate}

By the measure convergence $\mu_{\wt{T}_i}\to\mu_T$, $\forall\, \e>0$ there exists $i_0$ such that $\forall\, i\ge i_0$:
\[\l_i^{-n}\mu_{T_i}(B^{n+1}_{\l_i}(x_i))=\mu_{\wt{T}_i}(B^{n+1}_1(0))\le|\spt T\cap B^{n+1}_1(0)|+\e\]

Furthermore because of \ref{Hausdorff close} of Lemma \ref{graph current convergence to a plane} for any $\e>0$, there exists $i_0$ such that for all $i\ge i_0$
\begin{equation*}\label{Ti e-close to T}
\frac{1}{\l_i}\sup_{y\in B^{n+1}_{\l_i}(x_i)\cap\spt T_i}\dist(y-x_i,\spt T)\le\sup_{y\in B^{n+1}_1(0)\cap\spt\wt{T}_i}\dist(y,\spt T)<\e
\end{equation*}
and if $\partial T\ne 0$, we also have that
\begin{equation*}\label{Ti e-close to T boundary}
\frac{1}{\l_i}\sup_{y\in B^{n+1}_{\l_i}(x_i)\cap\Phi_i}\dist(y-x_i,\spt\partial T)\le
\sup_{y\in B^{n+1}_{1}(0)\cap\wt{\Phi}_i}\dist(y,\spt\Phi)<\e
\end{equation*}
since $\wt{\Phi}_i\cap (B_{r/(2\l_i)}^n(0)\times\R)\toap \Phi$.

Hence taking $P$ to be the $n$-dimensional linear subspace that contains the support of $T$ we get a contradiction.

In the special case when $x_i\in\Phi_i$ we argue in the same way. In this situation, for the limit $T$ we are in case (ii) $\partial T=[\![\Phi]\!]\ne0$ and furthermore $0\in\Phi$. Hence we get a contradiction by taking $P_+=\spt T$.
\end{proof}

\section{Approximating the MCE}\label{dirichlet problem section}
Throughout this section we let $\W$ be a $C^{1,\a}$ bounded domain in $\R^n$ and $\Phi$ a compact, embedded $C^{1,\a}$ submanifold of $\partial\W\times\R$, such that for a sequence $\phi_i\in C^{1,\a}(\partial{\W})$, $\graph \phi_i\toa\Phi$, where the convergence is as in Definition \ref{c1a convergence}. By translating $\W$ we can assume that $0\in\partial\W$ and hence for some $r>0$, $(\W,\Phi)\in{\B}^\a_r$, with ${\B}^\a_r$ as in Definition \ref{ar regular class}.
We let also $H=H(x',x_{n+1})$ be a $C^1$ function in $\ov{\W}\times\R$, which is non decreasing in the $x_{n+1}$-variable and such that $\|H\|_0\le n \o_n^{1/n}|\W|^{-1/n}$.

In this section we will show that the Dirichlet problem of prescribed mean curvature equal to $H$ (cf. \ref{PMCE}) and with boundary data $(\W,\Phi)$, can be approximated by a sequence of new Dirichlet problems for the prescribed mean curvature equation which have the form of the one defined in \ref{PMC dirichlet problem with divergence term} of Section \ref{section the DP with RD}. We will construct the new equations in such a way that
\begin{enumerate}
\item[(a)] We have uniform $C^{1,\a}$ bounds for the graphs of the solutions of the approximating problems 
\item[(b)] We can construct barriers for the solutions and prove gradient bounds and hence existence of the solutions 
\end{enumerate}

\subsection*{Constructing the approximating sequence}\label{constructing approximation}

For $(\W,\Phi)\in{\B}^\a_r$, let $\{\W_k\}$ be a sequence of bounded, $C^\infty$ domains with $\W_k\subset\W$ for all $k$, $\phi_k\in C^\infty(\partial{\W}_k)$ and $\Phi_k=\graph\phi_k$ be such that $(\W_k,\Phi_k)\in{\B}^\a_{r}$ and 
\[ \W_k\toa \W\,\,,\,\,\Phi_k\toa \Phi\]
with the convergence being as in Definition \ref{c1a convergence}.

For each $k$ we consider the following Dirichlet problem
\begin{equation}\label{approx prob}
\begin{split}
\sum_{i=1}^nD_i\left(\frac{D_iu_k}{\sqrt{1+|Du_k|^2}}\right)&=\sum_{i=1}^nD_if^i_k+H_k \text{  in  } \W_k\\
u_k&=\phi_k \text{  on  }\partial\W_k
\end{split}
\end{equation}
where the equation above is to be interpreted weakly (as is \ref{MCE*}) and $H_k:\ov{\W}_k\times\R\to\R$, $f_k=(f_k^1,\dots, f_k^n):\ov{\W}_k\times\R\to\R$ satisfy the following properties for a sequence $\d_k\downarrow 0$:
\begin{enumerate}
\item[(i)] $H_k=H_k(x',x_{n+1})$ is a $C^1$ function in $\ov{\W}_k\times\R$, which is non decreasing in the $x_{n+1}$-variable and such that $\|H_k\|_0\le\|H\|_0$ and
\[H_k(x',x_{n+1})=\begin{cases}H(x',x_{n+1})&\text{for } x'\in \W_k: \dist(x',\partial\W_k)>{2\d_k}\\ 
0&\text{for } x'\in \W_k: \dist(x',\partial\W_k)<\d_k.\end{cases}\]
\item[(ii)]  There exists a neighborhood $V_k$ of $\partial\W_k$ in $\W_k$ such that 
\[\{x'\in\W_k:\dist(x',\partial\W_k)<\d_k\}\subset V_k\,\,,\,\, B^n_{r/4}(x')\subset V_k\,\, \forall x'\in\partial\W_k\cap B^n_r(0)\] and such that $f_k$ is $C^{0,\a}$ when restricted in $V_k\times\R$ and in particular it satisfies the estimate
\[\|f_k\|_{0, B_{r/4}^{n}(x')\times\R} + r^\a[f_k]_{\a, B_{r/4}^{n}(x')\times\R}
 \le C\left(\|\eta_k\|_{0,(\partial\W_k\cap B_r^{n}(x'))\times\R}+ r^\a[\eta_k]_{\a, (\partial\W_k\cap B_r^{n}(x'))\times\R}\right)\]
for all $x'\in\partial\W_k\cap B^n_r(0)$. Here $\eta_k$ is the inward pointing unit normal to the cylinder $\W_k\times\R$ and $C$ is a constant that depends only on $n$, in particular it is independent of $k$ (recall that $r$ is such that $(\W,\Phi)\in{\B}^\a_{r}$). Also
\[\sum_{i=1}^nD_if_k(x',x_{n+1})=\begin{cases}\dvg \eta_k(x',x_{n+1}) &\text{  for  } x'\in\partial\W_k, x_{n+1}>\phi_k(x')\\
-\dvg\eta_k(x',x_{n+1})&\text{  for  }x'\in\partial\W_k, x_{n+1}<\phi_k(x')\\
0&\text{  for  }x'\in\W_k:\dist(x',\partial\W_k)>\d_k \,\,\,(\text{weakly}).\end{cases}\]
\end{enumerate}

We now show how to construct $H_k, f_k$ satisfying the above properties.

For any $\d>0$ we define $\W_k^\d$ to be the $\d$-neighborhood of $\partial\W_k$ in $\W_k$, i.e.
\[\W_k^\d=\{x'\in\W_k:\dist(x',\partial\W_k)<\d\}.\]

Let $\{\d_k\}$ be a sequence such that 
\begin{equation}\label{dk}
\d_k\to 0\,\, \text{ and }\,\,\d_k^{1/2}\|H_{\partial\W_k}\|_0\to 0
\end{equation}
 where $H_{\partial\W_k}$ denotes the mean curvature of $\partial\W_k$ with respect to the inward pointing unit normal.
We also take $\d_k$ small enough so that the nearest point projection, which we will denote by $\proj_{\partial\W_k}(x')$, is well defined for all $x'\in \W^{2\d_k}_k$. Notice that we can do this since $\partial\W_k$ is $C^\infty$ and using \ref{dk} we have that
\begin{equation}\label{projprop}
|\proj_{\partial\W_k\times\R}(x)-\proj_{\partial\W\times\R}(y)|\le C|x-y|\,\,\forall\,x,y\in\W_k^{2\d_k}\times\R
\end{equation}
where $C$ is a constant that is independent of $k$. This enables us to extend $\eta_k$ in $\W_k^{2\d_k}\times\R$ by letting $\eta_k(x',x_{n+1})=\eta_k(\proj_{\partial\W_k}(x'),x_{n+1})=\eta_k(\proj_{\partial\W_k}(x'),0)$ and for this extension, using \ref{projprop} we have that
\begin{equation}\label{etaext}
[\eta_k]_{\a, (\W^{\d_k}_k\cap B_r^{n}(x'))\times\R}\le C[\eta_k]_{\a, (\partial\W_k\cap B_r^{n}(x'))}\,\,\forall\,x'\in\partial\W_k
\end{equation}
where $C$ is a constant independent of $k$. 
Similarly we can extend ${\phi}_k$ in $\W_k^{2\d_k}$ by ${\phi}_k(x')=\phi_k(\proj_{\partial\W_k}(x'))$. Furthermore we pick the sequence $\{\d_k\}$ so that
\begin{equation}\label{dk2}
\d_k^{\a/2}\|D\phi_k\|_0\to 0.
\end{equation}
We remark that this is a technical assumption that will be used later for proving global gradient estimates for a solution of \ref{approx prob} (cf. Lemma \ref{global gradient bounds lemma}).

With $\d_k$ as above, we let
\[H_k(x',x_{n+1})=\begin{cases} H(x',x_{n+1}) &\text{  in  } (\W_k\setminus\W_k^{2\d_k})\times\R\\
0 &\text{  in  } \W_k^{\d_k}\times\R\end{cases}\]
and extend it in the rest of the domain $\W_k\times \R$ so that it is $C^1$, non-decreasing in the $x_{n+1}$-variable and so that $\|H_k\|_{0,\overline{\W}_k\times\R}\le\|H\|_{0,\overline{\W}\times\R}$. Hence we have constructed $H_k$, satisfying the properties described in (i) above.

To construct $f_k$, we define $U^+, U^-\subset \W^{\d_k/2}_k\times\R$ by
\begin{equation}\label{U}
\begin{split}U^+&=\{(x',x_{n+1}):x'\in\W_k^{\d_k/2}, x_{n+1}\ge\phi(x')+\dist(x',\partial\W_k\times\R)\}\\
U^-&=\{(x',x_{n+1}):x'\in\W_k^{\d_k/2}, x_{n+1}<\phi(x')-\dist(x',\partial\W_k\times\R)\}.\end{split}
\end{equation}

By Lemma \ref{match divergence}, Remark \ref{rmk match div},  there exists a smooth vector field\\ $X=(X^1,\dots,X^n, X^{n+1})$ in $U^-$, independent of the $x_{n+1}$-variable, such that 
\[\dvg X=\sum_{i=1}^n D_i X^i=0\,\,,\,\,X (x',x_{n+1})=2\eta_k(x',x_{n+1}) \text{ for  }x'\in\partial\W_k\]
 and 
 \[\begin{split}\|X\|_{0, (B_r^n(x')\times\R)\cap U^-}&+r^\a [X]_{\a, (B_r^n(x')\times\R)\cap U^-}\\&\le C\left(\|\eta_k\|_{0, (B_r^n(x')\cap \partial\W_k)\times\R}+ r^\a[\eta_k]_{\a, (B_r^n(x')\cap \partial\W_k)\times\R}\right)\end{split}\]
for any $x'\in \partial\W_k$ and where $C$ is independent of $k$ (for sufficiently large $k$).

Using again Lemma \ref{match divergence}, Remark \ref{rmk match div}, there exists a neighborhood $V_k$ of $\partial\W_k$ in $\W_k$ such that
\[\W_k^{\d_k}\subset V_k\,\,,\,\, B^n_{r/4}(x)\subset V_k\,\, \forall x\in\partial\W_k\cap B^n_r(0)\]
and a smooth vector field $Y=(Y^1,\dots,Y^n, Y^{n+1})$ in $(V_k\setminus\W^{\d_k}_k)\times\R$, independent of the $x_{n+1}$ variable, such that 
\[\dvg Y=\sum_{i=1}^n D_i Y^i=0\,\,,\,\,Y (x',x_{n+1})=\eta_k(x',x_{n+1}) \text{ for  }x'\in\partial V_k\setminus\partial\W_k\]
 and 
 \[\begin{split}\|Y\|_{0, (B_r^n(x')\cap (V_k\setminus\W^{\d_k}_k))\times\R}&+r^\a [Y]_{\a, (B_r^n(x')\cap (V_k\setminus\W^{\d_k}_k))\times\R}\\&\le C\left(\|\eta_k\|_{0, (B_r^n(x')\cap \partial\W_k)\times\R}+ r^\a[\eta_k]_{\a, (B_r^n(x')\cap \partial\W_k)\times\R}\right)\end{split}\]
for any $x'\in \partial\W_k$ and where $C$ is independent of $k$ (for sufficiently large $k$).

We then define $f_k:U=U^+\cup U^-\cup((V_k\setminus\W^{\d_k}_k)\times\R)\to\R^n$ as follows: For each $i\in\{1,\dots,n\}$ we let
\begin{equation}\label{f}f_k^i(x)=\begin{cases}\eta_k^i(x)& \text{ if }x\in U^+\\
X^i-\eta_k^i(x)& \text{ if }x\in U^-\\
Y^i(x)& \text{ if }x\in (V_k\setminus\W^{\d_k}_k)\times\R\\
0& \text{ if }x\in (\W_k\setminus V_k)\times\R.\end{cases}\end{equation}
One can easily check now, using the estimates for the norms of $X$ and $Y$, the definition of $U^\pm$ as well as \ref{projprop}, \ref{etaext}, that $f_k=(f_k^1,\dots, f_k^n)$ is a $C^{0,\a}$ vector field in the domain $U=U^+\cup U^-\cup((V_k\setminus\W^{\d_k}_k)\times\R)$ and
 \begin{equation}\label{oaest}\begin{split}\|f_k\|_{0, (B_{r/4}^n(x')\times\R)\cap U}&+r^\a [f_k]_{\a, (B_{r/4}^n(x')\times\R)\cap U}\\&\le C\left(\|\eta_k\|_{0, (B_r^n(x')\cap \partial\W_k)\times\R}+ r^\a[\eta_k]_{\a, (B_r^n(x')\cap \partial\W_k)\times\R}\right)\end{split}\end{equation}
 for any $x'\in \partial\W_k\cap B^n_r(0)$ and where $C$ is independent of $k$. Hence we can extend $f_k$ in $\W_k\times\R$ so that the estimate \ref{oaest} still holds (with $U$ replaced by $\W_k\times\R$).

\begin{rmk}\label{lpoff}
By the construction of $f_k$ and using \emph{Lemma \ref{match divergence}}, \emph{Remark \ref{rmk match div}} we note that $\sum_{i=1}^nD_if^i_k$ is well defined and smooth in $(\ov{\W}_k\times\R)\setminus\graph\phi_k|_{\partial\W_k}$ and in this domain 
\[\left\|\sum_{i=1}^n D_if^i_k\right\|_0\le c\|\dvg\eta_k\|_0\Rightarrow\d_k^{1/2}\left\|\sum_{i=1}^n D_if^i_k\right\|_0\stackrel{k\to\infty}{\longrightarrow}0\] 
since $\d_k^{1/2}\|H_{\partial\W_k}\|_0\to 0$ \emph{(cf. \ref{dk})}.

Furthermore, since $f_k$, as defined in \emph{\ref{f}}, is independent of the $x_{n+1}$-variable in each of the domains $U^+$, $U^-$ and $(V_k\setminus\W_k^{\d_k})$, we can extend it in $\W_k\times\R$ so that it is still independent of the $x_{n+1}$ variable in each of the domains 
\[\{(x',x_{n+1}):x'\in\W_k^{\d_k}, x_{n+1}\ge\phi(x')+\dist(x',\partial\W_k\times\R)\}\]
and
\[\{(x',x_{n+1}):x'\in\W_k^{\d_k}, x_{n+1}<\phi(x')-\dist(x',\partial\W_k\times\R)\}.\]
In these domains we also have that $\sum_{i=1}^n D_if_k^i(x)$ is equal to $\dvg\eta_{\partial\W_k}(x)$ and $-\dvg\eta_{\partial\W_k}(x)$ respectively. This extra property of $f_k$ will be used later for proving global gradient estimates for a solution of \emph{\ref{approx prob}} \emph{(cf. Lemma \ref{global gradient bounds lemma})}.

\end{rmk}

\begin{rmk}\label{supbound}
The solutions $u_k$ of the approximating problems satisfy a uniform sup estimate, i.e. if $u_k\in C^{1,\a}(\W_k)$ are solutions of the problems \emph{\ref{approx prob}}, then
\[\|u_k\|_0\le M\]
for some constant $M$ independent of $k$.

To see this note that by the assumption on $\|H\|_0$ and by \emph{Remark \ref{lpoff}}, for $\d_k$ small enough
\[\left|\int_\W (H_k+\sum_{i=1}^n D_if^i_k)\zeta d\H^n\right|\le (1-\e_0)\int_\W|D\zeta|d\H^n\]
for all $\zeta\in C^1_0(\W)$ and where $\e_0<1$ is a constant independent of $k$. Hence we get a uniform sup estimate \cite[pg. 408]{GT}. 
 \end{rmk}

\subsection*{$C^{1,\a}$ regularity of the approximating graphs}\label{regularity subsection}

In the following theorem, which is essentially an application of Theorem \ref{close-to-plane thm} and Allard's regularity theorem \cite{Afirstvarofvarifold}, we prove that the graphs of the solutions are close to planes in uniform sized balls.

\begin{theorem}\label{graph uniformly close to a plane}
For each $k$ let $u_k\in C^{1,\a}(\ov{\W}_k)$ be a (weak) solution of \emph{\ref{approx prob}}. 
For any $\e>0$, there exists $\l_0=\l_0(\e)$ such that the following holds:

For any $k$, $x_k\in\graph u_k\cap (B^n_{r/8}(0)\times\R)$ and $\l\le\l_0$ 
\begin{equation*}\label{P-close}
\l ^{-1}\sup_{\graph u_k\cap B^{n+1}_{\l }(x_k)}\dist(y-x_k,P)<\e.\tag{(1)}
\end{equation*}
for some $n$-dimensional linear subspace $P=P(x_k,\e,\l)$.

In particular if $x_k\in\graph \phi_k\cap (B^n_{r/8}(0)\times\R)$ then \emph{\ref{P-close}} holds with an $n$-dimensional halfspace $P_+=P_+(x_k,\e,\l)$ in place of $P$, such that $0\in\partial P_+$ and
\begin{equation*}\label{P-close bdry}
\l^{-1}\sup_{\Phi_k\cap B^{n+1}_{\l }(x_k)}\dist(y-x_k,\partial P_+)<\e.
\tag{(2)}
\end{equation*}
\end{theorem}

\begin{proof}

We assume that the conclusion is not true. Then for some $\e_0>0$ and for any $\l_0>0$ there exist  $k_j$, $x_j\in\graph u_{k_j}\cap(B^n_{r/8}(0)\times\R)$ and $\l_j<\l_0$ such that the conclusion of the lemma for $k=k_j$, $x_k=x_j$ and $\l=\l_j$ fails. Hence there exist sequences $\{k_j\}$, $\{x_j\}$ such that $x_j\in\graph u_{k_j}$ and a sequence $\{\l_j\}\downarrow 0$ such that the conclusion \ref{P-close} of the lemma with this $\e_0$ and with $k=k_j$, $x_k=x_j$, $\l=\l_j$ fails for all $j$.

Since for all $k$, $u_k\in C^{1,\a}(\ov{\W}_k)$, we can assume that $k_j\to\infty$. Hence without loss of generality we can take $k_j=j$.

Let $d_j=\dist(x_j,\partial\W_j\times\R)$. Standard PDE theory implies uniform interior $C^{1,\a}$ estimates for the solutions of the of the problems \ref{approx prob}, therefore we can assume that $d_j\to 0$ (cf. Remark \ref{change of basis rmk}).

In the special case when $x_j\in\Phi_j$, we can apply Theorem \ref{close-to-plane thm} with $x=x_j$, $\l=\l_j$. Hence for any $\e>0$ there exists $j_0$ such that for $j\ge j_0$
\[\l_j^{-1}\sup_{\graph u_j\cap B^{n+1}_{{\l}_j}({x}_j)}\dist(y-{x}_j,P_+)
\le \e\]
for some $n$-dimensional linear halfspace $P_+$ with $0\in\partial P_+$, such that
\[\l_j^{-1}\sup_{\Phi_j\cap B^{n+1}_{{\l}_j}({x}_j)}\dist(y-{x}_j,\partial P_+)
\le \e\]
and so by taking $\e=\e_0$ we get a contradiction, which proves the special case \ref{P-close bdry} of the theorem. 

We assume now that $x_j\notin\Phi_j$.
Applying Theorem \ref{close-to-plane thm} with $x=x_j$ and $\l=d_j+\l_j$ we get that
for any $\e>0$ there exists $j_0$ such that for $j\ge j_0$
\begin{equation*}\label{thm1}
(d_j+\l_j)^{-1}\sup_{\graph u_j\cap B^{n+1}_{d_j+{\l}_j}({x}_j)}\dist(y-{x}_j,P)
\le \e
\tag{(3)}\]
for some $n$-dimensional linear subspace $P$ and
\begin{equation*}\label{thm2}
\o_n^{-1}(d_j+\l_j)^{-n}|\graph u_j\cap B^{n+1}_{d_j+{\l}_j}({x}_j)|
\le 1+\e.
\tag{(4)}\]

We will consider two different cases, namely $\liminf d_j^{-1}\l_j>0$ or $d_j=0$ and $\liminf d_j^{-1}\l_j=0$ and show that in both cases we are led to a contradiction.

\textit{{Case 1:}} $\liminf d_j^{-1}\l_j>0$ or $d_j=0$. 

In this case \ref{thm1} impies that 
\begin{equation*}\label{ineq1}
\l_j^{-1}\sup_{\graph u_j\cap B^{n+1}_{{\l}_j}({x}_j)}\dist(y-{x}_j,P)\le\frac{\l_j+d_j}{\l_j}\e\le c\e.
\end{equation*}
Hence by taking $\e$ small enough, so that $c\e<\e_0$, where $c$ is as in the above inequality, we get a contradiction.

\textit{{Case 2:}} $\liminf d_j^{-1}\l_j=0$.

In this case for any $p\ge 2n$
\begin{equation*}\label{bdedmc}
d_j^{1-n/p}\left(\int_{B^{n+1}_{d_j}(x_j)}\left|{H}_j+\sum^{n}_{i=1} D_i f^i_j\right|^pd\mu_{T_j}\right)^\frac 1p\stackrel{j\to\infty}{\longrightarrow} 0
\tag{(5)}
\end{equation*}
since either $d_j>2\d_j$, which implies that $B^{n+1}_{d_j/2}(x_j)\subset \{x\in\W_j\times\R:\dist(x,\partial\W_j\times\R)>\d_j\}$ where $\sum_{i=1}^n D_if^i_j=0$, or  $d_j\le2\d_j$ in which case \ref{bdedmc} is true because of Remark \ref{lpoff}.

Furthermore \ref{thm2} implies that
\[\o_n^{-1}d_j^{-n}|\graph u_j\cap B^{n+1}_{{d}_j}({x}_j)|\le\left(1+\frac{\l_j}{d_j}\right)^n(1+\e).\]

Hence for any $\e'>0$ there exists $j_0$ such that for all $j\ge j_0$, $\graph u_j\cap B^{n+1}_{ d_j}(x_j)$ satisfies the hypothesis of Allard's interior regularity theorem and thus there exists $\theta\in(0,1)$ such that $\graph u_j\cap B^{n+1}_{\theta d_j}(x_j)$ is the graph of a $C^{1,\a}$ function $v_j$ above an $n$-dimensional linear space $P$, with  the $C^{1,\a}$ norm of $v_j$ is less than $\e'$. Hence for all $j\le j_0$ such that $\l_j<\theta d_j$
\[\l_j^{-1}\sup_{\graph u_j\cap B^{n+1}_{\l_j}(x_j)}\dist(y-x_j, P)\le \e'\]
which for $\e'=\e_0$ gives a contradiction.

\end{proof}

We will show next that the graphs of the solutions $u_k$ are not only $\e$-close to planes, as we proved in Theorem \ref{graph uniformly close to a plane}, but in fact they are $\e$-close in the $C^{1,\a}$ sense, i.e. we will prove that around each point there exists a uniform sized ball in which $\graph u_k$ is a $C^{1,\a}$ manifold with uniformly bounded $C^{1,\a}$ norm.

\begin{rmk}\label{change of basis rmk}
For all $k$, $u_k\in C^{1,\a}(\ov{\W}_k)$ and so $\graph u_k$ is a $C^{1,\a}$ manifold-with-boundary equal to $\Phi_k=\graph\phi_k$. Therefore given $\e_0\in(0,1/2)$, for any $k$ and $x\in\graph u_k$ there exists some $\r=\r(k, x,\e_0)$ such that 
\begin{equation*}\label{nua}
\r^\a\frac{|\nu_k(y)-\nu_k(z)|}{|y-z|^\a}\le\e_0/4\,\,\forall\,y,z\in B^{n+1}_\r(x)\cap\graph u_k
\tag{(1)}
\end{equation*}
where for any point $x=(x', u_k(x'))\in\graph u_k$, $\nu_k(x)$ is the downward pointing unit normal of $\graph u_k$ at $x$.
Note that provided $\dist(x,\partial\W\times\R)\ge d>0$, the radius $\r$ satisfying \emph{\ref{nua}} is independent of $ k$ and $x$, i.e. there exists $\r_0=\r_0(d,\e_0)<d$ such that the inequality in \emph{\ref{nua}} holds with any $k$ and $x\in\graph u_k$ such that $B_\r^{n+1}(x)\subset (\W_k\setminus\W^d_k)\times\R$, recall that $\W_k^d=\{x\in\W_k:\dist(x,\partial\W_k)<d\}$. That is because standard PDE estimates \cite[Chapter 13]{SimonKorevaar, GT}
 imply that for any $d>0$ we have that
\begin{equation*}\label{PDE remark}
\sup _{\{k:\d_k<d/2\}}\|u_k\|_{1,\a,\W_k\setminus\W_k^d}<C(d)
\end{equation*}
where $C(d)$ is a constant independent of $k$.

For any $k$, $x\in\graph u_k$ and $\r=\r(k,x,\e_0)$ satisfying \emph{\ref{nua}}, we have that 
\begin{equation*}\label{identi}
\graph u_k\cap B^{n+1}_\r(x)=\graph v\cap B^{n+1}_\r(x)\tag{(2)}
\end{equation*}
with $v\in C^{1,\a}((x+(L_x\cap U))\cap B^{n+1}_\r(x);L_x^\perp)$ and $L_x= T_x\graph u_k$, the tangent space of $\graph u_k$ at $x$. Since $v(0)=0$, $|Dv(0)|=0$ and for all $x=(x', v(x'))\in\graph u_k$ we have that
\[ \nu_k(x)=\left(\frac{D_1v(x')}{\sqrt{1+|Dv(x')|^2}},\dots, \frac{D_nv(x')}{\sqrt{1+|Dv(x')|^2}}, -\frac{1}{\sqrt{1+|Dv(x')|^2}}\right)\]
 it is easy to check that \emph{\ref{nua}} implies that
 \[\|v\|_{1,\a}\le\e_0.\]
  Also $U$ is a $C^{1,a}$ domain of $L_x$, since either
 \begin{itemize}
 \item[i.] $\Phi_k\cap B^{n+1}_\r(x)=\emptyset$, in which case $U=L_x$ or
 \item[ii.]  $x+(\partial U\cap B^{n+1}_\r(0))=\proj_{T_x\graph u_k}(\Phi_k\cap B^{n+1}_\r(x))$.
\end{itemize} 

Furthermore the function $v$ satisfies the equation
\[\sum_{i=1}^nD_i\left(\frac{D_iv}{\sqrt{1+|Dv|^2}}\right)=\dvg {f}_k+{H}_k\text{  in  }{U}\cap B^n_\r(0)\]
 where for $x'\in \W_k$, $y'\in x+(L_x\cap U)$ we have identified $(x', u_k(x'))$ with $(y', v(y'))$ using \emph{\ref{identi}}.
 
 Given $\e_0$, $k$ and $x\in\graph u_k$, let $\r$ be such that \emph{\ref{nua}} holds and assume furthermore that for some $\e<\e_0$, $B_{{\r}}^{n+1}(x)\cap\graph u_k$ is $\e$-close to some $n$-dimensional linear space $P$, i.e. 
\[{\r}^{-1}\sup_{y\in B^{n+1}_{{\r}}(x)\cap\graph u_k}\dist(y-x, P)<\e.\]
We then have that
\[\r^{-1}\dist(P\cap B_\r(0), T_x\graph u_k\cap B_\r(0))\le\e+\e_0\]
and it is then easy to check \emph{(}by writing $P$ as the graph of a linear function above $T_x\graph u_k$\emph{)} that this last inequality implies that
\begin{equation*}\label{enn}
\|N-\nu_k(x)\|<5\e_0/2
\end{equation*}
where $N$ is the normal to $P$. Using this and \emph{\ref{nua}} we have that
\begin{equation*}\label{NN}
\|\nu_k(y)- N\|\le\|\nu_k(y)- \nu_k(x)\|+\|\nu_k(x)-N\|<3\e_0\,\,\forall\,y\in \graph u_k\cap B^{n+1}_{\r}(x).
\tag{(3)}
\end{equation*}
This implies that  \emph{\ref{identi}} holds with $v\in C^{1,\a}((x+(P\cap \wt U))\cap B^{n+1}_\r(x);P^\perp)$, such that $\|v\|_{1,\a}\le 6\e_0$ and where $\wt U$ is a $C^{1,a}$ domain of $P$.

 \end{rmk}

\begin{theorem}\label{theta is bounded below}
Let $u_k\in C^{1,\a}(\ov{\W}_k)$ be a solution of \emph{\ref{approx prob}}. For $0<\e_0<1/4$ there exists a constant $\r_0=\r_0(\e_0)$, independent of $k$, such that 
\[\r_x=\sup_r\{\k(\graph u_k,x,r)<\e_0\}<\r_0\]
for all $x\in\graph u_k\cap B^n_{r/16}(0)$, where $\k$ is as in \emph{Definition \ref{c1a convergence}}.
\end{theorem}
\begin{proof}
Let $\s=r/8$. By Theorem \ref{graph uniformly close to a plane} we have that for any $\e>0$ there exists $\l_0=\l_0(\e)$ such that for any $k$, $x_k\in\graph u_k\cap (B^n_\s(0)\times\R)$ and $\l\le\l_0$ 
\begin{equation*}\label{hypth}
\l^{-1}\sup_{\graph u_k\cap B^{n+1}_\l(x_k)}\dist(y-x_k,P)<\e.
\tag{(1)}
\end{equation*}
for some $n$-dimensional linear subspace $P=P(x_k,\e,\l)$.

We fix a $k$, and define the following:
\[d(x)=\dist(x,\partial B^n_\s(0)\times\R)\]

\[\theta_1=\min\left\{\frac{\r_{x}}{d(x)}: x\in\graph u_k\cap \Phi_k\cap({B}^n_{\s}(0)\times\R)\right\}\]
\[\theta_2=\min\left\{\frac{\r_{x}}{d(x)}: x\in(\graph u_k\setminus\Phi_k) \cap({B}^n_{\s}(0)\times\R)\right\}.\]
Note that both these minima are attained. Given $\e$ small enough (that will be determined later), let $\l_0=\l_0(\e)$ be such that \ref{hypth} holds. We can assume that
\begin{equation*}\label{mintheta}
\min\{\theta_1,\theta_2\}\le\frac18\min\{1/2,\l_0/\s\}
\tag{(2)}
\end{equation*}
since otherwise for all $x\in\graph u_k\cap (B_{\s/2}^n(0)\times\R)$
\[\r_{x}\ge\min\{\theta_1,\theta_2\} d(x)\ge \frac18\min\{\s/4,\l_0/2\}\]
and hence the lemma is trivially true.

Let $x\in\graph u_k \cap (\ov{B}^n_{\s}(0)\times\R)$ be such that the following holds
\begin{itemize}
\item[(i)] If $\theta_1\le 4\theta_2$ then $x\in\Phi_k$ and $\r_x=\theta_1 d(x)$
\item[(ii)]  If $\theta_1>4\theta_2$ then $x\notin\Phi_k$ and $\r_x=\theta_2 d(x)$
\end{itemize}

In both cases, using \ref{mintheta}, we have that
\[\r_{x}\le4\min\{\theta_1,\theta_2\} d(x)\le \l_0/2.\]
Therefore there exists an $n$-dimensional linear subspace $P_0$ such that
\begin{equation*}\label{Pclo}
(2\r_x)^{-1}\sup_{\graph u_k\cap B^{n+1}_{2\r_x}(x)}\dist(y-x,P_0)<\e.
\tag{(3)}
\end{equation*}
Remark \ref{change of basis rmk}, the definition of $\r_x$ and \ref{Pclo} imply that \begin{equation*}\label{voverp0}
\graph u_k\cap B^{n+1}_{\r_x}(x)=\graph v\cap B^{n+1}_{\r_x}(x)
\tag{(4)}
 \end{equation*}
 where $v\in C^{1,\a}((x+(P_0\cap U))\cap B^{n+1}_{\r_x}(x); P^\perp_0)$ is such that $\|v\|_{1,\a}<6\e_0$ and where $U$ is a $C^{1,\a}$ domain of $P_0$, provided that $2\e<\e_0/4$.
 
We will show that we can extend $v$ in $B^{n+1}_{(1+\gamma)\r_x}(x)$ for some $\gamma\in (0,1)$ such that \ref{voverp0} still holds with $(1+\gamma)\r_x$ in place of $\r_x$ and $\|v\|_{1,\a}< c\e_0$ for some constant $c$.

Let $z\in B_{2\r_x}^{n+1}(x)\cap \graph u_k$. Then because of the choice of $x$ and using \ref{mintheta} we have that
\[\r_z\ge d(z)\min\{\theta_1,\theta_2\}\ge \frac14 \r_x\frac {d(z)}{d(x)}\ge \frac14 \r_x\left(1-\frac{2\r_x}{d(x)}\right)\ge\r_x/8.\]
Furthermore by \ref{Pclo}
\[(\r_x/8)^{-1}\sup_{\graph u_k\cap B^{n+1}_{\r_x/8}(z)}\dist(y-z,P_0)<16\e.\]
Hence, by Remark \ref{change of basis rmk}, $B^{n+1}_{\r_x/8}(z)\cap\graph u_k$ can be written as the graph of a $C^{1,\a}$ function above $z+P_0$ with $C^{1,\a}$ norm less that $6\e_0$, provided that $16\e<\e_0/4$. Since $\dist(z-x, P_0)<2\e\r_x$, by a translation we have that
 \begin{equation*}\label{vzoverp0}
\graph u_k\cap B^{n+1}_{\r_x/8}(z)=\graph v_z\cap B^{n+1}_{\r_x/8}(z)
\tag{(5)}
 \end{equation*}
 where  $v_z\in C^{1,\a}(x+(P_0\cap U_z);P_0^\perp)$ is such that 
  $\|v_z\|_{1,\a}<7\e_0$ and where $U_z$ is a $C^{1,\a}$ domain of $P_0$.

Note that for $z\in\graph u_k\cap\partial B^{n+1}_{\r_x}$, because of \ref{Pclo}, we have that
 \[\graph v_z\cap  B^{n+1}_{\r_x/8}(z)\cap B^{n+1}_{\r_x}(x)\ne\emptyset\]
 and so
 \[\graph v_z\cap \graph v \ne\emptyset\]
where $v$ is as defined in \ref{voverp0}. Therefore, using the graphical representations in \ref{vzoverp0} for any $z\in\graph u_k\cap\partial B^{n+1}_{\r_x}$, we can extend the function $v$ so that $v\in C^{1,\a}((x+(P_0\cap U))\cap B^{n+1}_{(1+1/8)\r_x}(x);P_0^\perp)$ and with
\begin{equation*}\label{extv}
\|v\|_{1,\a}<c\e_0
\tag{(6)}
\end{equation*}
 where $c$ is an absolute constant. For $\gamma\in(0,1/16)$, we can check (using again \ref{Pclo}) that for any $y\in\graph u_k\cap B^{n+1}_{(1+\gamma)\r_x}(x)$ there exists $z\in\graph u_k\cap\partial B_{r_x}^{n+1}(x)$ such that $|z-y|<\r_x/8$. Hence for the extended function $v$ we have that
\begin{equation*}\label{idi}
\graph u_k\cap B^{n+1}_{(1+\gamma)\r_x}(x)=\graph v\cap B^{n+1}_{(1+\gamma)\r_x}(x).\tag{(7)}
\end{equation*}

Furthermore $v$ satisfies the following equation: 
\[D_i\left(\frac{D_iv}{\sqrt{1+|Dv|^2}}\right)=\dvg {f}_k+{H}_k\text{  in  }U\cap B^{n}_{(1+\gamma/2)\r_x}(0).\]
where as in  Remark \ref{change of basis rmk}, for $x'\in \W_k, y'\in x+P_0$ we identify $(x', u_k(x'))$ with $(y', v(y')$ using \ref{idi}.

For this function $v$ if either $U=\R^n$ or $0\in\partial U$ we can apply the interior or boundary $C^{1,\a}$ Schauder estimates respectively, in $B^{n}_{\r_x}(0)\subset B^{n}_{(1+\gamma/2)\r_x}(0)$, which imply that
\begin{equation*}\label{schauder estimates}
\|v\|_{1,\a,B^n_{\r_x}(0)}\le C\left(\e+\r_x^\a[{f}_k]_{\a,B^n_{(1+\gamma/2)\r_x}(0)}+\r_x\|{H}_k\|_{0,B^n_{(1+\gamma/2)\r_x}(0)}\right)
 \tag{(8)}
\end{equation*}
where $C$ is a constant depending only on $\a,n,\gamma$. In this case \ref{schauder estimates} implies a lower bound for $\r_x$.
To see this note that the LHS is bounded below by $c^{-1} \e_0$, where $c$ is the absolute constant in \ref{extv}, since if it wasn't true then (by \ref{extv}) we would have that $\|v_k\|_{1,\a,B^n_{(1+\gamma/2)\r_x}(0)}\le\e_0$, which would contradict the definition of $\r_x$. Hence taking $\e$ small enough, the inequality gives a lower bound on $\r_x$, that is independent of $k$.

To finish the proof we need to show that we can indeed apply the Schauder estimates, i.e. we need to show that either $0\in\partial U$ or $U=\R^n$.

If $x\in\Phi_k$ then $0\in\partial U$. So we can assume that $x\notin\Phi_k$ which implies that $\theta_1>4\theta_2$ (i.e. we are in case (ii), as described at the beginning of the proof). We will show that $\Phi_k\cap B^{n+1}_{(1+\gamma)\r_x}(x)=\emptyset$ and hence $U=\R^n$. Assume that for some $\ov{x}\in\Phi_k$, $|x-\ov{x}|<2\r_x$. Then
\[\frac{\r_{\ov{x}}}{d(\ov{x})}\ge\theta_1>4\theta_2=4\frac{\r_x}{d({x})}\Rightarrow\]
\[\r_{\ov{x}}>4\left(1-\frac{|x-\ov{x}|}{d(x)}\right)\r_x\ge4(1-2\theta_2)>\frac72\r_x\]
where we have used \ref{mintheta}. Hence $B^{n+1}_{3\r_x/2}(x)\subset B^{n+1}_{\r_{\ov{x}}}(\ov{x})$. But this would contradict the definition of $\r_x$ and so $\dist(x,\Phi_k)>2\r_x$.
\end{proof}

\begin{rmk}\label{rmk on 1a regularity}
\emph{Theorem \ref{theta is bounded below}} implies that for any $\e>0$ there exists $\r_0=\r_0(\e)$
such that for all $k$
\[\|\nu_k(x)-\nu_k(y)\|\le \e|x-y|^\a\,,\,\forall\,x,y\in\graph u_k\cap (B^n_{r/16}(0)\times\R): |x-y|<\r_0\]
where $\nu_k$ denotes the downward unit normal of $\graph u_k$.

Recall that $\partial\W$, $\Phi$ are compact, $C^{1,\a}$ embedded submanifolds, which along with \emph{Remark \ref{change of basis rmk}} imply that for any $\e>0$ there exists $\wt\r_0=\wt\r_0(\e)$
such that for all $k$
\[\|\nu_k(x)-\nu_k(y)\|\le \e|x-y|^\a\,,\,\forall\,x,y\in\graph u_k: |x-y|<\wt{\r}_0\]
where $\wt{\r}_0$ now also depends on 
\[\sup_r\{r:\k(\partial\W,x,r)<1,\forall x\in\partial\W\text{  and  }\k(\Phi,x,r)<1,\forall x\in\Phi\}.\]
Let $M=\sup_{\W_k}|u_k|$ and recall that $M$ is independent of $k$ \emph{(cf. Remark \ref{lpoff})}.
Covering $\W_k\times[-M,M]$, by balls of radius $\wt\r_0$ we conclude that
\[\|\nu_k(x)-\nu_k(y)\|\le C|x-y|^\a\,,\,\forall\,x,y\in\graph u_k\]
where $C$ does not depent on $k$.
\end{rmk}

\subsection*{Gradient estimates for the solutions to the approximating problems \ref{approx prob}}\label{Grad est section}

Our goal here is to show a priory $C^{1,\a}$ estimates for the solutions $u_k\in C^{1,\a}(\ov{\W}_k)$ of \ref{approx prob}. That will allow us (cf. Theorem \ref{existence lemma}) to apply the Leray-Schauder theory to prove existence of such  solutions.

We first show that for each $k$, we have boundary gradient estimates for a solution $u_k$, by using local barriers at each boundary point. In particular we have the following:
\begin{lemma}\label{boundary gradient estimates lemma}
Let $u_k\in C^{1,\a}(\ov{\W}_k)$ be a solution of \emph{\ref{approx prob}} then
\[\|Du_k\|_{0,\partial\W_k}\le C\]
where $C$ depends on $\d_k$, $\|u_k\|_0$, $\|H_{\partial\W_k}\|_0$ and $\|\phi_k\|_2$.
\end{lemma}
\begin{proof}

For any $x'_0\in\partial\W_k$ let $N=B^n_r(x'_0)\cap\W^{\d_k/2}_k$ and let ${\phi}^1_k,\phi^2_k:N\rightarrow \R$ be $C^2$ functions that satisfy the following
\[{\phi}^1_k(x'_0)=\phi^2_k(x'_0)=\phi_k(x'_0)\]
\[{\phi}^1_k(x')\ge\phi_k(x')\ge \phi^2_k(x') \hspace{.5cm}\forall x'\in \partial N\cap\partial\W_k\]
\[{\phi}^1_k(x')\ge \|u_k\|_0\,\,,\,\,{\phi}^2_k(x')\le-\|u_k\|_0\hspace{.5cm} \forall x'\in\partial N\setminus\partial\W_k\]
and
\[{\phi}^1_k(x')>\phi_k(x')+\dist(x',\partial\W_k)\,\,,\,\,{\phi}^2_k(x')<\phi_k(x')-\dist(x',\partial\W_k)\hspace{.5cm}\forall x'\in N\]
so that
\[(x',{\phi}^1_k(x'))\in U^+\,\,,\,\,(x',{\phi}^2_k(x'))\in U^-\hspace{.5cm}\forall x'\in N\]
where $U^\pm$ are as defined in \ref{U}, so that
\[\sum_{i=1}^nD_i f_k^i=\begin{cases}\dvg\eta_k &\text{ in }U^+\\
-\dvg\eta_k &\text{ in }U^-\end{cases}\]
where $\eta_k$ is the inward pointing unit normal to $\partial\W_k\times\R$, extended in $\W_k^{\d_k}\times\R$, so that $\eta_k(x', x_{n+1})=\eta_k(\proj_{\partial\W_k}(x'),x_{n+1})$.

We look at the following Dirichlet problems:
\begin{equation*}\label{new problem}
\begin{split}
\sum_{i=1}^nD_i\left(\frac{D_i{u}^1_k}{\sqrt{1+|D{u}^1_k|^2}}\right)&=\dvg{\eta}_k \text{  in  } N\\
{u}^1_k&=\phi^1_k \text{  on  }\partial N
\end{split}
\tag{(1)}
\end{equation*}
\begin{equation*}\label{new problem'}
\begin{split}
\sum_{i=1}^nD_i\left(\frac{D_i{u}^2_k}{\sqrt{1+|D{u}^2_k|^2}}\right)&=-\dvg{\eta}_k \text{  in  } N\\
{u}^2_k&=\phi^2_k \text{  on  }\partial N.
\end{split}
\tag{(2)}
\end{equation*}

 By standard PDE theory \cite[Theorem 14.6]{GT} we know that there exists a positive function $\psi:\R\rightarrow\R_+$, such that for the functions defined by 
\[\psi^+(x')={\phi}^1_k(x')+\psi(d(x'))\,\,,\,\,\psi^-(x')=\phi^2_k(x')-\psi(d(x'))\]
where $d(x')=\dist(x',\partial\W_k)$, we have that
\[\sum_{i=1}^nD_i\left(\frac{D_i\psi^+}{\sqrt{1+|D\psi^+|^2}}\right)\le\dvg{\eta}_k\text{  on  }N\]
and
\[\sum_{i=1}^nD_i\left(\frac{D_i\psi^-}{\sqrt{1+|D\psi^-|^2}}\right)\ge-\dvg{\eta}_k\text{  on  }N.\]
So $\psi^+$ is an upper barrier for a solution $u^1_k\in C^{1,\a}(\ov{\W})$ of the problem \ref{new problem} at the point $x'_0$,  $\psi^-$ is a lower barrier for a solution ${u}^2_k\in C^{1,\a}(\ov{\W})$ of the problem \ref{new problem'} at the point $x'_0$ and their gradients satisfy an estimate
\begin{equation*}\label{barest}
|D\psi^+(x'_0)|, |D\psi^-(x'_0)|\le C(\|H_{\partial\W_k}\|_0,\|{\phi}_k\|_2,{\d}_k^{-1}\|u_k\|_0),
\tag{(3)}
\end{equation*}
a constant depending on $\|H_{\partial\W_k}\|_0,\|{\phi}_k\|_2$ and ${\d}_k^{-1}\|u_k\|_0$.

We claim that $\psi^+$ and $\psi^-$ are also an upper and respectively a lower barrier for $u_k$ at $x'_0$. 

Let
\[N^+=\{x'\in N: u_k(x')\ge \psi^+(x')\}\,\,,\,\,N^-=\{x'\in N: u_k(x')\le \psi^-(x')\}.\]
Since $\psi^+(x')\ge u_k(x')\ge\psi^-(x')$ for all $x'\in\partial N$, we have that 
\[u_k(x')=\psi^+(x') \text{ on }\partial N^+\,\,\text{  and  }\,\,u_k(x')=\psi^-(x') \text{ on }\partial N^-.\] 
Furthermore since $\psi^+(x')\ge {\phi}^1_k(x')$, $\psi^-(x')\le\phi_k^2(x')$, we have that 
\[(x',u_k(x'))\in U^+\,,\,\forall x'\in N^+\,\,\text{ and }\,\,(x',u_k(x'))\in U^-\,,\,\forall x'\in N^-.\]
Hence
\[\sum_{i=1}^nD_i\left(\frac{D_iu_k}{\sqrt{1+|Du_k|^2}}\right)=\sum_{i=1}^nD_i{f}^i_k=\dvg{\eta}_k\text{  in  }N^+\]
and
\[\sum_{i=1}^nD_i\left(\frac{D_iu_k}{\sqrt{1+|Du_k|^2}}\right)=\sum_{i=1}^nD_i{f}^i_k=-\dvg{\eta}_k\text{  in  }N^-.\]
Thus, by the comparison principle we have that $u_k(x')=\psi^+(x')$ for all $x'\in N^+$, $u_k(x')=\psi^-(x')$ for all $x'\in N^-$ and so $\psi^+$ and $\psi^-$ are upper and respectively lower barriers for $u_k$.
\end{proof}

We claim now that the boundary gradient bounds (Lemma \ref{boundary gradient estimates lemma}) along with the $C^{1,\a}$ estimates that we have shown for the graph of $u_k$ as a submanifold (Theorem \ref{theta is bounded below}), imply global gradient bounds for the function $u_k$.

\begin{lemma}\label{global gradient bounds lemma}
Let $u_k\in C^{1,\a}(\ov{\W}_k)$ be a solution of \emph{\ref{approx prob}} then
\[\sup_{\ov{\W}_k}|Du_k|\le C\]
where $C$ depends on $\|Du_k\|_{0,\partial\W_k}$, $\sup_{\W_k} |H_k|+|DH_k|$, $\|\phi_k\|_2$ and $\sup_{\partial\W_k} |H_{\partial\W_k}|+|DH_{\partial\W_k}|$.
\end{lemma}
\begin{proof}

Recall that by the construction of the approximating problems, in the domain
\[((\W_k\setminus\W_{k}^{\d_k})\times\R)\cup\{(x',x_{n+1}):x'\in\W^{\d_k}_k, |x_{n+1}-\phi_k(x')|>\dist(x',\partial\W_k)\}\] 
the mean curvature of $\graph u_k$, $H_k+\sum D^i f^i_k$ is smooth and its derivative with respect to the $x_{n+1}$-variable is non-negative (cf. Remark \ref{lpoff}). Hence, by standard gradient estimates for a solution to the prescribed mean curvature equation \cite{SimonKorevaar}, it suffices to show that that for any $x'_0\in\partial\W_k$
 
 \begin{equation*}\label{refgrad}
 |Du_k(x')|\le C\,\,,\,\,\,\forall\, x'\in B_{2\d_k}^n(x'_0): |u_k(x')-\phi_k(x')|<2\dist(x',\partial\W_k)
  \tag{(1)}
 \end{equation*}
for some constant $C$ as in the statement of the lemma.

Let $x_0=(x'_0, u_k(x'_0))\in\Phi_k$. Recall that we picked $\d_k$ so that we can extend $\phi_k$ in $\W^{2\d_k}$ by letting $\phi_k(x')=\phi_k(\proj_{\partial\W_k}(x'))$ and by \ref{projprop} we have that for any $x',y'\in B^n_{2\d_k}(x_0')$
\[|\phi_k(x')-\phi_k(y')|\le C\d_k\|D\phi_k\|_0\]
where $C$ is a constant independent of $k$. We have also picked $\d_k$ so that $\d_k^{\a/2}\|D\phi_k\|_0\to0$ (cf. \ref{dk2}), hence for $\d_k$ small enough: If
$x'\in B^n_{2\d_k}(x_0')$  and  $|u_k(x')-\phi_k(x')|<2\dist(x',\partial\W_k)$, then $(x',u_k(x'))\in B_{\d_k^{1/2}}^{n+1}(x_0)$.

By the uniform $C^{1,\a}$ estimates for $\graph u_k$ (cf. Theorem \ref{theta is bounded below}, Remark \ref{rmk on 1a regularity}), there exists $\r_0$ such that for all $k$
\begin{equation*}\label{ro}
\|\nu_k(x)-\nu_k(y)\|\le |x-y|^\a,\,\,\forall\,x,y\in\graph u_k: |x-y|\le \r_0
\tag{(2)}
\end{equation*}
where $\nu_k$ is the downward unit normal to $\graph u_k$.

Let $K=1 +\sup_k\|D\phi_k\|_0$. We will show that \ref{refgrad} holds for all $\d_k$ small enough so that $\d_k^{1/2}<\r_0$ and
\begin{equation*}\label{dktec}
\d_k^{\a/2}<\frac{1}{2\sqrt{1+4K^2}}\Leftrightarrow 4\d_k^\a+16(\d_k^{\a/2}K)^2<1.
\tag{(3)}
\end{equation*}
We will consider two different cases:

{\it Case 1:} $|Du_k(x'_0)|\le 2K$. Then for any $x=(x',u(x'))\in B^{n+1}_{\d_k^{1/2}}(x_0)$ we have, by \ref{ro}, \ref{dktec}:
\begin{align*}
\frac{1}{\sqrt{1+|Du_k(x')|^2}}>\frac{1}{\sqrt{1+|Du_k(x'_0)|^2}}-\d_k^{\a/2}\Rightarrow |Du(x')|\le 4(1+K)
\end{align*}
in which case \ref{refgrad} holds.

{\it Case 2:} $|Du(x'_0)|> 2K$. Then, by \ref{ro}, \ref{dktec}, for any $x=(x',u(x'))\in B^{n+1}_{\d_k^{1/2}}(x_0)$ we have that
\begin{align*}
\frac{1}{\sqrt{1+|Du_k(x')|^2}}<\frac{1}{\sqrt{1+|Du_k(x'_0)|^2}}+\d_k^{1/2}\Rightarrow |Du(x')|>1+K.
\end{align*}
Hence in this case 
\[ (B^{n+1}_{\d_k^{1/2}}(x_0)\cap\graph u_k\cap(B^n_{2\d_k}(x_0')\times\R))\setminus\Phi_k\subset\{(x',x_{n+1}):|x_{n+1}-\phi_k(x')|\ge2d(x)\}\]
and so \ref{refgrad} is trivially true.
\end{proof}

\subsection*{Existence of solution to the approximating problems \ref{approx prob}} \label{existence of solutions section}

Lemma \ref{global gradient bounds lemma} and standard applications of the De Giorgi, Nash, Moser theory give the following $C^{1,\a}$ estimates \cite[Theorem 13.2]{GT}.
\begin{cor}\label{global C1a bound lemma}
Let $u_k\in C^{1,\a}(\ov{\W}_k)$ be a solution of \emph{\ref{approx prob}} then
\[\|u_k\|_{1,\a,\ov{\W}_k}\le C\]
where $C$ depends on $\d_k$, $\sup_{\W_k} |H_k|+|DH_k|$, $\|\phi_k\|_2$ and $\sup_{\partial\W_k} |H_{\partial\W_k}|+|DH_{\partial\W_k}|$.
\end{cor}

We can now prove the existence of a solution to the problem \ref{approx prob}.

\begin{theorem}\label{existence lemma}
The Dirichlet problem \emph{\ref{approx prob}} has a solution in $C^{1,\a}(\ov{\W}_{k})$.
\end{theorem}
\begin{proof}
Let $p\ge n/(1-\a)$. We define the family of operators
$$T_\s:C^{1,\a}(\ov{\W}_k) \rightarrow W^{2,p}(\W_k)\,,\,\s\in[0,1]$$
such that for any $v \in C^{1,\a}(\ov{\W}_k)$, $u=T_\s v$ is defined to be the solution of the
linear Dirichlet problem
\begin{equation*}\label{linear eqn}\begin{split}
\sum_{i,j=1}^na_{ij}(x',Dv) D_{ij}u&= g_\s(x',v) \text{ in }\W_k \\u&=\s \phi_k \text{ on }
\partial \W_k\end{split}
\tag{(1)}
\end{equation*}
where
\begin{equation*}
a_{ij}(x',p)=\frac{\delta_{ij}}{\sqrt{1+|p|^2}}-\frac{p_ip_j}{\left(\sqrt{
1+|p|^2}\right)^3}
\end{equation*}
and
\[g_\s(x',x_{n+1})=\s H_k(x',x_{n+1})+\sum_{i=1}^n D_i {f}^i_{k,\s}(x',x_{n+1})\]
where the vector field ${f}_{k,\s}$ is constructed in the same way as $f_k$ was constructed in the beginning of this section (cf. construction under equation \ref{approx prob}), but with boundary data $\s\phi_k$ instead of $\phi_k$ and $\d_k$ replaced by $s(\s)\d_k$, where $s:[0,1]\to [0,1]$ is a continuous increasing function such that $s(1)=1$ and $s(0)=0$. Note that then $[{f}_{k,\s}]_{\a}\le c[\eta_k]_{\a}$, where $\eta_k$ is the inward pointing unit normal to $\partial\W_k\times\R$.

{\textit{Claim:}}
For all $\s\in[0,1]$, $T_\s$ is well defined, compact and continuous.

	$a_{ij}\in C^0(\ov{\W}_k)$ and $g_\s(x', v(x'))$ is bounded and thus in $L^p(\W_k)$. Hence, there exists a solution $T_\s v=u\in W^{2,p}(\W_k)$ of \ref{linear eqn} \cite[Theorem 9.18]{GT} and by the Calderon-Zygmund inequality, for this solution we have that
\begin{equation*}\label{W2p estimate}
\|u\|_{W^{2,p}}\le C (\|\phi_k\|_{W^{2,p}}+\|g_\s\|_p).
\tag{(2)}
\end{equation*}
Therefore $T_\s$ is well defined.

Assume now that $\{v_i\}$ is a sequence of functions in $C^{1,\a}({\ov{\W}_k})$, such that for some $K>0$
\[\| v_i\|_{1,\a,\ov{\W}_k}\le K\,,\,\forall i.\]
Then, by the Arzela-Ascoli theorem, after passing to a subsequence, $v_i\to v\in C^{1,\a}(\ov{\W}_k)$, where the convergence is with respect to the $C^{1,\a'}$ norm, for all $\a'<\a$.

Let $u_i=T_\s v_i$. Then, because of \ref{W2p estimate} and by the Sobolev embedding theorem \cite[Theorem 7.26]{GT}, after passing to a subsequence, $u_i\to u$ with respect to the $C^{1,\a'}$ norm for all $\a'<\a$ and where $u\in W^{2,p}(\W_k)$. Furthermore $u$ is a
solution to the equation
$$a_{ij}(x',Dv) D_{ij}u=g_\s(x',v) \text{ in }\W_k $$
and since $u_i=\s\phi_k$ on $\partial \W_k$ for all $ i$, we have that $u=\s\phi_k$ on $\partial \W_k$.
Hence
$$T_\s v=u$$ 
and thus $T_\s$ is compact and continuous.

If $u_\s$ is a fixed point of the operator $T_\sigma$, then $u_\s$ is a solution of the following Dirichlet problem
\begin{equation*}\label{sigma prob}\begin{split}\sum_{i=1}^nD_i\left(\frac{D_iu_\s}{\sqrt{1+|Du_\s|^2}}\right)&=\sum_{i=1}^n D_i{f}^i_{k,\s}(x',u_\s)+\s H_k(x',u_\s) \text{ in }\W_k\\
 u&=\sigma\phi_k\text{ on }\partial\W_k.\end{split}\tag{(3)}
 \end{equation*}
 
 Since $0$ is the unique solution for $T_0(0)$ and a fixed point of $T_1$ corresponds to a solution of the problem \ref{approx prob}, the Leray-Schauder theory implies that \ref{approx prob} is solvable in $C^{1,\a}(\ov{\W}_k)$ if there exists a constant $C$ such that
\begin{equation*}\label{unif}
\|u_\s\|_{1,\a}\le C\,\,,\,\,\forall\s\in[0,1]
\tag{(4)}
\end{equation*}
where for each $\s$, $u_\s\in C^{1,\a}(\ov{\W}_k)$ is a solution to \ref{sigma prob}. 

Since $\|u_\s\|_0$ are uniformly bounded (cf. Remark \ref{supbound}) and because of standard PDE estimates \cite[Theorem 13.2]{GT}, for proving \ref{unif}, it suffices to show that
\begin{equation*}\label{sind}
\|Du_\s\|_{0}\le C \,\,,\,\,\forall \s\in[0,1].
\tag{(5)}
\end{equation*}

Note first that the results of this section concerning the regularity of the approximating graphs of the solutions to \ref{approx prob} (Theorems \ref{graph uniformly close to a plane}, \ref{theta is bounded below}) are applicable for the family of problems given in \ref{sigma prob} for $\s\in[0,1]$. Hence Theorem \ref{theta is bounded below} implies uniform (independent of $\s$) $C^{1,\a}$ estimates for the graphs of the solutions $u_\s$ as manifolds. Furthermore for each $u_\s$, Lemma \ref{boundary gradient estimates lemma} implies a boundary gradient estimate $\|Du_\s\|_{0,\partial\W_k}\le C_\s$, possibly depending on $\s$ and consequently we can apply Lemma \ref{global gradient bounds lemma} to get a global gradient estimate for each $u_\s$ (depending on $C_\s$). 

We will show that by choosing the function $s(\s)$ appropriately, $C_\s$ is in fact independent of $\s$, which would imply \ref{sind}. 

By the construction of the barriers in Lemma \ref{boundary gradient estimates lemma} (cf. estimate \ref{barest} in proof of Lemma \ref{boundary gradient estimates lemma}, \cite[Chapter 14]{GT}) we note that it suffices to show that $s(\s)^{-1}\|u_\s\|_{0}$ are uniformly bounded. The following sup estimate shows that this can be achieved as long as we take $s=s(\s)=\s^{1/2}$.

Let $v=(u_\s-l)_+$, where $l=\sup_{\partial\W_k}\|\s\phi_k\|_0$. For some $\e_1,\e_2\in(0,1)$, let also
\[\W_1=(\W_k\setminus\W_k^{s\d_k})\cap\{x:|Du_\s(x)|>\e_1\}\,\,,\,\, \W_2=(\W_k\setminus\W_k^{s\d_k})\cap\{x:|Du_\s(x)|\le\e_1\}\] 
and
\[\W^s_1=\W^{s\d_k}_k\cap\{x:|Du_\s(x)|>\e_2\}\,\,,\,\, \W^s_2=\W_k^{s\d_k}\cap\{x:|Du_\s(x)|\le\e_2\}.\] 
By using $v$ as a test function in the weak form of the equation \ref{sigma prob} we have:
\[\begin{split}\frac{\e_1}{\sqrt{1+\e_1^2}}\int_{\W_2}|Dv|d\H^n&+\frac{1}{\sqrt{1+\e_1^2}}\int_{\W_2}|Dv|^2d\H^n\\
&+\frac{\e_2}{\sqrt{1+\e_2^2}}\int_{\W^s_2}|Dv|d\H^n+\frac{1}{\sqrt{1+\e_2^2}}\int_{\W^s_1}|Dv|^2d\H^n
\le\\
C\s\bigg(\int_{\W_2}|Dv|d\H^n&+\int_{\W_1}|Dv|^2d\H^n+ 1\bigg)\\&+Cs^{1/2}\left(\int_{\W^s_2}|Dv|d\H^n+\int_{\W^s_1}|Dv|^2d\H^n+ s\right)\end{split}\]
where $C$ is a constant depending only on $\|H_k\|_0, |\W_k|$ (cf. Remark \ref{supbound}). So for $\e_1=\s^{1/2}$, $\e_2=4Cs^{1/2}$ (where $C$ is as in the above estimate) and for $s,\s$ small enough and we get the following:
\[\s^{1/2}\int_{\W_2}|Dv|d\H^n+s^{1/2}\int_{\W^s_2}|Dv|d\H^n\le C(\s+s^{1+1/2})\Rightarrow\]
\[\s^{1/2}\int_{\W_2\cup\W_1}|Dv|d\H^n+s^{1/2}\int_{\W^s_2\cup\W_1^s}|Dv|d\H^n\le C(\s+s^{1+1/2})+C(\s+s^2)\Rightarrow\]
\[\int_{\W_k}|Dv|d\H^n\le C(\s^{1/2}+\s^{-1/2}s^{1+1/2})\]
and by the Sobolev inequality
\[\left(\int_{\W_k}|v|^\frac{n}{n-1}d\H^n\right)^\frac{n-1}{n}\le C\s^{-1/2}(\s+s^{1+1/2}).\]
This implies a sup estimate (cf. \cite[Section 10.5] {GT}):
\[\|v\|_0\le C s\]
where $C$ is independent of $s,\s$, provided that $\s^{1/2}\le s$.

\end{proof}

\section{Main Theorem and applications}\label{main theorem chapter}
As in Section \ref{dirichlet problem section}, we let $\W$ be a $C^{1,\a}$ bounded domain in $\R^n$ and $\Phi$ a compact, embedded $C^{1,\a}$ submanifold of $\partial\W\times\R$. We will use the following notation:

The set $(\partial\Omega\times\R)\setminus\Phi$ is the
union of two disjoint open connected components $U_{\Phi},V_{\Phi}$, where
$U_{\Phi}\supset \{(x',x_{n+1}):x'\in
\partial\Omega,\, x_{n+1}>R\}$ and $V_{\Phi}\supset \{(x',x_{n+1}):x'\in
\partial\Omega,\,x_{n+1}<-R\}$ for sufficiently large $R$. We can think of these components as
the parts of the cylinder $\partial\W\times\R$ that lie ``above" and ``below" $\Phi$
respectively. Then for any
multiplicity one $n$-current $S$ with $\spt S\subset\ov{\W}\times\R$ and $\partial S=[\![\Phi]\!]$, there exists a
multiplicity one $(n+1)$-current, which we will denote by $\widetilde{S}$, such that
\begin{equation}\label{tilde}
S=[\![V_\Phi]\!]+\partial\widetilde{S}\,\,,\,\,\spt\wt{S}\subset\W\times\R.
\end{equation}

We now state our Main Theorem:

\begin{theorem}[Main Theorem]\label{main theorem}
Given $\W$ a $C^{1,\a}$ bounded domain of $\R^n$, $\Phi$ a compact, embedded $C^{1,\a}$ submanifold of $\partial\W\times\R$, such that for a sequence $\phi_i\in C^{1,\a}(\partial{\W})$, $\graph \phi_i\toa\Phi$ and $H=H(x',x_{n+1})$ a $C^1$ function in $\ov{\W}\times\R$, which is non decreasing in the $x_{n+1}$-variable and such that $\|H\|_0\le n \o_n^{1/n}|\W|^{-1/n}$, we let $u$ be a function in $BV(\Omega)$ that minimizes the functional
\[\mathcal{F}(u)=\int_\Omega\sqrt{1+|Du|^2}dx'+\int_\Omega\int_0^{u(x)}H(x',x_{n+1})dx'dx_{n+1}+
  \lim_i\int_{\partial
  \Omega}|u-\phi_i|dx'.\]

Then for the current $T=[\![\graph u]\!]+Q$, where $Q$ is the multiplicity one
$n$-current such that $\spt Q \subset
\partial\Omega\times{\R}$ and $\partial
Q=[\![\Phi]\!]-[\![\trace u]\!]$, $\spt T$ is a $C^{1,\alpha}$ manifold-with-boundary, with boundary given by $\Phi$.

Moreover this current $T=T(\W,\Phi, H)$ locally minimizes the functional
\begin{equation*}\label{minimizing property of T}
\underline{\underline{M}}(T)+\int_{{\spt}\widetilde{T}}H(x',x_{n+1})dx'dx_{n+1}\tag{(1)}
\end{equation*}
among all $n$-currents with boundary $[\![\Phi]\!]$ and support in $\ov{\W}\times\R$.
\end{theorem}
\begin{proof}
Without loss of generality we can assume that $0\in\partial\W$. We also let $r>0$ be such that $(\W,\Phi)\in\B^\a_r$ (cf. Definition \ref{ar regular class}).

We use the approximating method described in section \ref{dirichlet problem section} with the given boundary data $(\W,\Phi)$. Let $u_k\in C^1(\ov{\W}_k)$ be the solutions to the approximating problems defined in \ref{approx prob}:
\begin{equation*}\label{napp}\begin{split}\sum_{i=1}^nD_i\left(\frac{D_iu_k}{\sqrt{1+|Du_k|^2}}\right)&=\sum_{i=1}^n D_if^i_k+H_k \text{  in  }\W_k\\
u_k&=\phi_k \text{  on  }\partial\W_k\end{split}
\tag{(2)}
\end{equation*}
where $\W_k\toa \W$ and $\Phi_k|_{\partial\W_k}\toa\Phi$ with $\Phi_k=\graph\phi_k$. Note that the solutions $u_k$ exist by Theorem \ref{existence lemma}.

Then, by Theorem \ref{theta is bounded below}, we have uniform $C^{1,\a}$ estimates for the graphs of $u_{k}$ (independent of $k$). In particular given $\e_0$, there exists $\rho$ such that 
\[\k(\graph u_{k},\r, x_k)<\e_0\,\,,\,\,\forall x_k\in\graph u_k\,\,\text{and}\,\,\forall k\]
where $\k$ is as in Definition \ref{c1a convergence}. 

Assume now that $B^{n+1}_{\r/2}(x)\subset\R^{n+1}$ is such that $B^{n+1}_{\r/2}(x)\cap \graph u_{k}\ne\emptyset$ for infinitely many $k$. Then for these $k$, $B^{n+1}_{\r/2}(x)\cap\graph u_{k}$ is the graph of a $C^{1,\a}$ function, of norm less than $\e_0$, above some $n$-dimensional affine space $P_k$. After passing to a subsequence $P_k\to P$, where $P$ is an $n$-dimensional affine space and hence
\begin{equation*}\label{videf}
\graph u_{k}\cap B^{n+1}_{\r/2}(x)=\graph v_k\cap B_{\r/2}^{n+1}(x)\tag{(3)}
\end{equation*}
where $v_k\in C^{1,\a}(P\cap U_k; P^\perp)$ is such that $\|v_k\|_{1,\a}<2\e_0$
 and
where $U_k$ is a $C^{1,\a}$ domain of $P$. Notice that either $U_k=P$ or else $\Phi_k\cap B_\r(x)\ne\emptyset$ in which case $\partial U_k\cap B^{n+1}_{\r/2}(x)=\proj_P(\Phi_{k}\cap B^{n+1}_{\r/2}(x))$. In the latter case, since $\Phi_{k}\toa\Phi$, we have that $U_k\toa U$, where $U$ is a $C^{1,\a}$ domain of $P$ such that $\partial U\cap B^{n+1}_{\r/2}(x)=\proj_P(\Phi\cap B^{n+1}_{\r/2}(x))$.

Hence we can apply the Arzela-Ascoli theorem to the sequence $\{v_k\}$, to conclude that, after passing to a subsequence,
\begin{equation*}\label{vdef}
v_k\to v
\tag{(4)}
\end{equation*}
with respect to the $C^{1,\a'}$ norm, for any $\a'<\a$, where $v\in C^{1,\a}(P\cap U; P^\perp)$ is such that $\|v\|_{1,\a}\le2\e_0$ and furthermore
\[\graph u_{k}\cap B^{n+1}_\r(x)\toap \graph v\cap B_\r^{n+1}(x).\]

Since $\|u_{k}\|_0$ are uniformly bounded, there exists a compact subset $D\subset\R^{n+1}$ such that $\graph u_{k}\subset D$ for all $k$. Covering $D$ with finitely many balls of radius $\r$ and applying the above discussion in each of them we get that after passing to a subsequence
\begin{equation*}\label{conv}\graph u_{k}\toap M
\tag{(5)}\end{equation*}
for all $\a'<\a$ and where $M$ is an embedded $C^{1,\a}$ manifold-with-boundary $\Phi$ and such that for any $x\in M$, $\k(M,\r,x)<2\e_0$.

For the sequence $\{u_{k}\}$, by standard PDE estimates (as discussed in Remark \ref{PDE remark})
we have uniform $C^{1,\a}$ estimates in compact sets of $\W$ (independent of $k$) and thus, after passing to a subsequence, $\{u_{k}\}$ converges with respect to the $C^{1,\a'}$ norm on compact sets of $\W$, for any $\a'<\a$, to a function $u\in C^{1,\a}(\W)$. Hence we have that $M\cap(\W\times\R)=\graph u$. Furthermore, since $u_{k}$ satisfy the equation in \ref{napp} and $\W_{k}\toa\W$, $u$ satisfies
\begin{equation*}\label{no bdry eqn}
\sum_{i=1}^nD_i\left(\frac{D_iu}{\sqrt{1+|Du|^2}}\right)=H(x',u(x'))\,,\,\text{in }\W.
\tag{(6)}
\end{equation*}

Let $T$ be the multiplicity one $n$-current such that $\spt T=M$. Then $\partial T=[\![\Phi]\!]$ and $T=[\![\graph u]\!]+Q$, where $Q$ is the multiplicity one
$n$-current such that $\spt Q\subset
\partial\Omega\times{\R}$ and $\partial
Q=[\![\Phi]\!]-[\![\trace u]\!]$.

$u$ satisfies equation \ref{no bdry eqn} and therefore $T$ minimizes the functional in \ref{minimizing property of T} of the theorem. To see this, let $W\subset\subset\overline{\Omega}\times\mathbb{R}$ and let $S$ be a multiplicity one $n$-current with boundary $[\![\Phi]\!]$ and such that $\spt T=\spt S$ outside $W$.

Then $T-S=\partial ({\wt{T}-\wt{S}})$, where $\wt T, \wt S$ are as defined at the beginning of this section (cf. equation \ref{tilde}) and note that $\wt{T}-\wt{S}$ has support in $W\cap(\W\times\R)$. Let
\[\omega=\sum_{i=1}^{n+1}(-1)^{i+1}e_i\cdot{\nu} dx_1\wedge\dots\wedge\widehat{dx_i}\wedge\dots\wedge
dx_{n+1}\] 
where $\nu$ is the downward pointing unit normal to $\graph u$ extended to be an
$\R^{n+1}$-valued function in $\W\times\R$ so that it is independent of the
$x_{n+1}$-variable. Then, because of the convergence in \ref{conv}, $\vec{T_{k}}\to\vec{T}$ (where $T_{k}=[\![\graph u_{k}]\!]$) with respect to the $C^{0,\a'}$ norm, for any $\a'<\a$. Hence $T(\o)=\mass(T)$ and arguing as in Lemma \ref{dvg-lemma} we have
\begin{equation*}\label{minarg}\begin{split}T(\o)- S(\o)=&(\wt{T}-\wt{S})(d\o)\Rightarrow \\
\mass(T)- \mass(S)
\le&-\int_{\spt \wt{T}\setminus\spt \wt{S}}H(x',x_{n+1})dxdx_{n+1}\\
&+\int_{\spt \wt{S}\setminus\spt \wt{T}}H(x',x_{n+1})dxdx_{n+1}\end{split}.\tag{(7)}\end{equation*}
Hence $T$ minimizes the functional defined in \ref{minimizing property of T} and so $u$ minimizes $\mathcal{F}$. 
\end{proof}

\begin{rmk}\label{uniqueness rmk}
If $S$ is another $n$-current with boundary $[\![\Phi]\!]$ and support in $\ov{\W}\times\R$ that minimizes the functional defined in \emph{\ref{minimizing property of T}} of \emph{Theorem \ref{main theorem}}, then $\vec{T}=\vec{S}$ almost everywhere. To see this note that by the argument \emph{\ref{minarg}} in the proof of \emph{Theorem \ref{main theorem}} we have that
 \[\mass(T)=\mass(S)=\int_{\spt S}<\vec{S},\vec{T}>d\mu_S.\]
Therefore $\spt S\cap(\W\times\R)=\graph u+c$, for some constant $c$ and hence $u+c$ (as well as $u$) minimizes the functional $\mathcal F$. Hence if $\trace u\cap\Phi\ne\emptyset$ then $S=T$ and $T$ is the unique current with this minimizing property. In particular we know that $u(x)=\phi(x)$ for any $x\in\partial\W$ such that $H_{\partial\W}(x)>|H(x,\phi(x))|$ \cite{MM-lpcdomaind-weakMSE}.
 \end{rmk}

\begin{rmk}\label{closrmk}
The regularity of $\spt T$, where $T=(\W,\Phi,H)$ \emph{(}as defined in \emph{Theorem \ref{main theorem}}\emph{)} depends on that of $\partial\Omega$, $\Phi$ and $H$ in a continuous way and the boundary regularity of $\spt T$ is a local result:

By the proof of \emph{Theorem \ref{main theorem}} we see that $\spt T$ can be approximated in the $C^{1,\a}$ sense \emph{(cf. Definition \ref{c1a convergence})} by a sequence of graphs of solutions $u_k$ to the approximating problems \emph{\ref{approx prob}} \emph{(}as described in \emph{Section} \emph{\ref{dirichlet problem section}}\emph{)}. Furthermore for this sequence we also have that $\|u_k-u\|_{1,\a', W}\to 0$ for all $W\subset\subset \W$, $\a'<\a$ and $\|u_k\|_{1,\a,W}\le C$, with the constant $C$ depending only on $\|H\|_{C^1}$, the $C^1$ norm of $H$. Hence $\|u\|_{1,\a,W}$ and therefore $\sup\{r: \k(\spt T,x,r)<\infty\}$ for any $x\in \spt T\cap (\W\times\R)$ depends only on $\|H\|_{C^1}$.

For points $x=(x',x_{n+1})\in\spt T\cap (\partial\W\times\R)$ we have \emph{(}by \emph{Theorem \ref{theta is bounded below}} and the proof of \emph{Theorem \ref{main theorem})} that  $\sup\{r: \k(\spt T,x,r)<\infty\}$ depends on $\|H\|_{C^1}$ but also on  $\sup\{r: \k(\partial\W,x',r)<\infty\}$ and $\sup\{r: \k(\Phi,x,r)<\infty\,\,,\,\forall x=(x',t)\in\Phi\}$. 
\end{rmk}

In the following corollary, which is an immediate consequence of Theorem \ref{main theorem}, we give some properties for the trace of the function $u$ that minimizes $\mathcal{F}$, as defined in Theorem \ref{main theorem}.
\begin{cor}\label{corollary about trace}
Let $(\W,\Phi)$, $H$, $u$ and $T=T(\W,\Phi,H)$ be as in the statement of \emph{Theorem \ref{main theorem}}.
\begin{enumerate}
\item[\emph{i.}] If $\trace u\cap\Phi\ne 0$, then $\forall\, x=(x',x_{n+1})\in\trace u
  \cap \Phi$ there exists $\rho>0$ such that either $B^{n+1}_\r(x)\cap V_{\Phi}=\emptyset$
  or $B^{n+1}_\r(x)\cap U_\Phi=\emptyset$. The radius $\rho$ depends only on  $\sup\{r: \k(\partial\W,x',r)<\infty\}$, $\sup\{r: \k(\Phi,x,r)<\infty\,\,,\,\forall x=(x',t)\in\Phi\}$ and $\|H\|_{C^1}$ \emph{(}because of \emph{Remark \ref{closrmk})}
  \item[\emph{ii.}]  If for some $x'\in\partial\W$, $x=(x',x_{n+1})\in \trace u\cap U_\Phi$ then
  $\vec{T}(x)$ coincides with the inward pointing unit normal of $\partial\W$ at
  $x'$ and if $x=(x',x_{n+1})\in \trace u\cap V_\Phi$ then
  $\vec{T}(x)$ coincides with the outward pointing unit normal of $\partial\W$ at
  $x'$.
  \end{enumerate}
  $U_\Phi, V_\Phi$ are as defined at the beginning of \emph{Section \ref{main theorem chapter}}.
\end{cor}

\subsection*{Higher Regularity}
In this paragraph we show higher regularity for $\spt T$, where $T$ is as in Theorem \ref{main theorem}, provided that we impose some additional regularity conditions on $\Phi,\partial\W$.

\begin{lemma}\label{W2pext}
If in addition to the hypotheses of \emph{Theorem \ref{main theorem}} we assume that $\partial\W$ and $\Phi$ are $W^{2,p}$ submanifolds, for some $p>n$, then for the current $T=T(\W,\Phi,H)$, as defined in \emph{Theorem \ref{main theorem}} we have that $\spt T$ is a $W^{2,p}$ manifold-with-boundary, with boundary given by $\Phi$.
\end{lemma}
\begin{proof}
Note first that since $\partial\W$, $\Phi$ are $W^{2,p}$, they are also $C^{1,\a}$ for $\a=1-n/p$ and hence (by Theorem \ref{main theorem}) $\spt T$ is $C^{1,\a}$.

Following the proof of Theorem \ref{main theorem} we can construct the vector fields $f^i_k$ in the approximating problems, as defined in \ref{napp} in the proof of Theorem \ref{main theorem}, so that 
\[\sup_{k}\|\sum_{i=1}^n D_i{f}^i_{k}\|_p<\infty \text{  and  } \sup_{k}\|{\phi}_{k}\|_{W^{2,p}}<\infty.\] 
Something which is possible because $\partial\W, \Phi$ are in $W^{2,p}$ (see also Lemma \ref{match divergence}, Remark \ref{rmk match div}).

Let $v_k$ be the local graphical representations of $\graph u_k$, as defined in \ref{videf} in the proof of Theorem \ref{main theorem}:   $v_k\in C^{1,\a}(P\cap U_k; P^\perp)$, for some $n$-dimensional affine space $P$
are such that 
\[\graph u_{k}\cap B^{n+1}_{\r/2}(x)=\graph v_k\cap B_{\r/2}^{n+1}(x).\]
Recall (Remark \ref{change of basis rmk}) that $v_k$ satisfy the following equations
\[\sum_{i=1}^nD_i\left(\frac{D_iv_k}{\sqrt{1+|Dv_k|^2}}\right)=\sum_{i=1}^n {f}^i_{k}+{H}^i_{k}\text{  in  }{U}_k\]
and if $\partial \Phi_k\cap B_{\r/2}^{n+1}(x)\ne\emptyset$ we also have that
\[v_k={\phi}_{k} \text{  on  }\partial U_k\cap B_{\r/2}^{n+1}(x).\]

Applying the Calderon-Zygmund inequality to the solutions $v_k$ and noticing that 
\[\sup_k\|\sum_{i=1}^nD_i {f}^i_{k}+{H}_{k}\|_p<\infty\,\,,\,\,\sup_k\|{\phi}_{k}\|_{W^{2,p}}<\infty\]
we conclude that $\|v_k\|_{W^{2,p}}$ are uniformly bounded. This implies that $v$ is in $W^{2,p}$, where $v$ is as in \ref{vdef} in the proof of Theorem \ref{main theorem} and in particular
\[\|v\|_{W^{2,p}}\le\sup_k\|v_k\|_{W^{2,p}}.\]
Hence $M=\spt T$ is a $W^{2,p}$ submanifold (see \ref{conv} in proof of Theorem \ref{main theorem}).
\end{proof}

Let $\W,\Phi, H, u, T=T(\W,\Phi, H)$ be as in Theorem \ref{main theorem}. Then by standard PDE theory, $\spt T\cap (\W\times\R)$ is a $C^{2}$ manifold. However, near points $x\in\trace u$ the best we can expect is that $\spt T$ is $C^{1,1}$. We will show that this is the case, provided that we impose higher regularity conditions on $\W, \Phi$. In particular we will show that around those points $\spt T$ can be expressed as the graph of a function that satisfies a variational inequality, an observation that for the case $H=0$ and for points $x\in\trace u\setminus\Phi$ was first made in \cite{Lin-Lau}. Thus using regularity results for such functions \cite{BK-variationalinequalities, GerhardtC11} we will show that $\spt T$ is a $C^{1,1}$ manifold-with-boundary provided that $\W$ is a $C^2$ domain and $\Phi$ is a $C^3$ manifold.

\begin{theorem}\label{gerhardthem}
If in addition to the hypotheses of \emph{Theorem \ref{main theorem}}, $\W$ is a $C^2$ domain and $\Phi$ is a $C^3$ submanifold of $\partial\W\times\R$, then for the current $T=T(\W,\Phi,H)$, as defined in \emph{Theorem \ref{main theorem}} we have that $\spt T$ is a $C^{1,1}$ manifold-with-boundary, with boundary given by $\Phi$.
\end{theorem}

\begin{proof}

Let $x=(x', x_{n+1})\in \trace u$. By Theorem \ref{main theorem} there exists $\r>0$ such that $B^{n+1}_\r(x)\cap\spt T$ can be represented as the graph of a $C^{1,\a}$ function above $P=T_x(\spt T)$, the tangent space of $\spt T$ at $x$, i.e.
\begin{equation*}\label{grrep}
B^{n+1}_\r(x)\cap\spt T=\graph v\cap B^{n+1}_{\r}(x)
\tag{(1)}
\end{equation*}
where $v\in C^{1,\a}((x+(P\cap U))\cap B^{n+1}_\r(x);P^\perp)$ is a $C^{1,\a}$ function and $U$ a $C^{1,\a}$ domain of $P$.

Assume that $P\ne T_x(\partial\W\times\R)$, the tangent space of $\partial\W\times\R$ at $x$. Then, because of Corollary \ref{corollary about trace} (ii), we have that $x\in \Phi$ and after replacing $\r$ with a smaller radius if necessary we have that
\[(B^{n+1}_\r(x)\cap\spt T)\setminus\Phi\subset\W\times\R\]
since $\spt T$ is $C^{1,a}$.
Hence $v$ satisfies 
\[\sum_{i=1}^n D_i\left(\frac{D_iv}{\sqrt{1+|Dv|^2}}\right)={H} \text{  in  }{U}\]
(cf. Remark \ref{change of basis rmk}). Also for $\partial U$ we have that $\graph{v}|_{x+\partial U}\cap B^{n+1}_\r(x)=\Phi\cap B^{n+1}_\r(x)$. Hence standard PDE estimates imply that $v\in C^{2}((x+(P\cap\ov{U}))\cap B^{n+1}_\r(x);P^\perp)$, provided that $\Phi$ is $C^{2,\a}$.

Hence we will assume that $P= T_x(\partial\W\times\R)$. In this case $v$ doesn't satisfy a prescribed mean curvature equation. However we will show that due to the minimizing property of $T$ (Theorem \ref{main theorem}), it satisfies a variational inequality. 

Note that we can also represent $(\partial\W\times\R)\cap B^{n+1}_\r(x)$ as the graph of a $C^{1,\a}$ function above $P=T_x(\partial\W\times\R)$, i.e. 
\[(\partial\W\times\R)\cap B^{n+1}_\r(x)=\graph\psi\cap B^{n+1}_\r(x)\]
for some $\psi\in C^{1,\a}((x+P)\cap B^{n+1}_\r(x);P^\perp)$.

In case $\Phi\cap B^{n+1}_\r(x)\ne\emptyset$, we define $f:(x+\partial U)\cap B^{n+1}_\r(x)\to P^\perp$ to be the restriction of $v$ on $\partial U$, so that
\[\Phi\cap B^{n+1}_\r(x)=\graph f\cap B_\r^{n+1}(x).\]

We also define the following set
\[\begin{split}\mathbf{K}=\bigg\{w\in C^{0,1}((x+(P\cap U))\cap B^{n+1}_\r(x);P^\perp):& w\ge \psi\,,\\ w=f\text{  on  }(x+\partial U)\cap B^{n+1}_\r(x)\,,&\,w=v\text{  on  }\partial B^{n+1}_\r(x)\cap(x+ U)\bigg\}.\end{split}\]

Then because of the minimizing property of $T$, $v$ minimizes the functional
\[B(v)=\int_{(x+U)\cap B^{n+1}_\r(x)}\sqrt{1+|Dv|^2}d\H^n+\int_{(x+U)\cap B^{n+1}_\r(x)}\int_{\psi(x')}^{v(x')}{H}(x',x_{n+1})dx_{n+1}dx'\]
among all functions in $\mathbf{K}$, where as usual for $x'\in\W$, $y'\in U$ we identify $(x', u(x'))$ with $(y', v(y'))$ using \ref{grrep}, here $u$ is such that $\spt T\cap (\W\times\R)=\graph u$.

Since $\mathbf{K}$ is a convex set, for any $w\in\mathbf{K}$, the function $\l:[0,1]\to B(v+\l(w-v))$, attains its minimum when $\l=0$, therefore 
\[\frac{d}{d\l}\bigg|_{\l=0} B(v+\l(w-v))\ge 0\Rightarrow\]
\[\int_{U\cap B^{n+1}_\r(x)}
\sum_{i=1}^n\frac{D_iv}{\sqrt{1+|Dv|^2}} D_i(w-v)+{H}(x',v(x'))(w-v) dx'\ge0\]
and hence $v$ satisfies the variational inequality
\begin{equation*}\label{varineq}
<Av+{H}v,w-v>\ge0\,\,\,,\,\,\,\,\forall\,w\in\mathbf{K}
\tag{(2)}
\end{equation*}
where
\[Av=-\sum_{i=1}^nD_i\left(\frac{D_iv}{\sqrt{1+|Dv|^2}}\right) \,\,\,\text{  and  }\,\,\, {H}v= {H}(x',v(x')).\]

Therefore by a theorem of Gerhardt \cite{GerhardtC11}, if $\psi$ is of class $C^2$ and $f$ is of class $C^3$ then $v$ is a $C^{1,1}$ function.

We remark here that in case $(x+U)\cap B^{n+1}_\r(x)=(x+P)\cap B^{n+1}_\r(x)$ (i.e. if $x\in\trace u\setminus\Phi$) then $v$ satisfies \ref{varineq}, but with the set $\mathbf{K}$ defined by
\[\mathbf{K}=\{w\in C^{0,1}((x+P)\cap B^{n+1}_\r(x);P^\perp): w\ge \psi\,,\,w=v\text{  on  }\partial B^{n+1}_\r(x)\}.\]
In this case, as was first shown in \cite{Lin-Lau}, we can also derive that $v$ is a $C^{1,1}$ function provided that $\psi$ is of class $C^2$
by a result in \cite{BK-variationalinequalities}.

\end{proof}
Finally we state a result about the regularity of the trace. It is known \cite{LS-bdry-regularity-onnonparametric-MS, FHLinNonparametric} that above a $C^4$ portion of $\partial\W$ where $H_{\partial\W}(x)<|H(x,\phi(x))|$, the trace of $u$ is a Lipschitz manifold. In the following Theorem we show that because of Theorem \ref{gerhardthem}, we can apply a result of Caffarelli \cite{Cafftraceresults} to show that it is actually $C^1$.

\begin{theorem}\
In addition to the hypotheses of \emph{Theorem \ref{main theorem}}, assume that $\W$ is a $C^4$ domain and $\Phi$ is a $C^3$ submanifold of $\partial\W\times\R$. Let $S=\{x'\in\partial\W: H_{\partial\W}(x')<|H(x',x_{n+1})|\,,\,\forall\,(x',x_{n+1})\in\Phi\}$. Then the trace of $u$ above $S$ is $C^1$, where $u\in \BV(\W)$ is the minimizer of $\mathcal{F}$, as in \emph{Theorem \ref{main theorem}}. 
\end{theorem}
\begin{proof}
Following the notation in the proof of Theorem \ref{gerhardthem}, we introduce the function $\ov{v}=v-\psi$ and define the set
\[F(\ov{v})=\partial\{y\in U\cap B^{n+1}_\r(x):\ov{v}(y)>0\}\cap\partial\{x\in U\cap B^{n+1}_\r(x):\ov{v}(y)=0\}.\]
Then $F(\ov{v})=\trace u\cap B^{n+1}_\r(x)$.
By \cite{LS-bdry-regularity-onnonparametric-MS, FHLinNonparametric}, $F(\ov{v})$ is a Lipschitz manifold. If we furthermore know that $\ov{v}$ is a $C^{1,1}$ function then we can apply the results in \cite{Cafftraceresults} to conclude that $F(\ov{v})$ is $C^1$.
This completes the proof, since if $\partial\W,\Phi\in C^3$, then by Theorem \ref{gerhardthem} $\ov{v}\in C^{1,1}$.
\end{proof}

\subsection*{ Corollaries, Applications}

We show that for $u\in BV(\W)$ minimizing the functional $\mathcal{F}$ (as in Theorem \ref{main theorem}), the trace of $u$ changes monotonely if we change the boundary data monotonely:

\begin{lemma}\label{T tilde is monotone}
Let $H,\W,\Phi_j, T_j=T(\W,\Phi_j,H)$, for $j=1,2$, be as in \emph{Theorem \ref{main theorem}} and such that $V_{\Phi_1}\subset V_{\Phi_2}$. Then
\[\spt\wt{T}_1\subset\spt\wt{T}_2\]
where $V_{\Phi_j}, \wt{T}_j$, $j=1,2$ are as defined at the beginning of section \emph{\ref{main theorem chapter}}, \emph{(cf. \ref{tilde})}.
\end{lemma}
\begin{proof} For $j=1,2$, we approximate $\spt T_j$ by graphs of solutions $u_k^j$ to the approximating problems  defined in \ref{approx prob} (as in \ref{napp} in the proof of Theorem \ref{main theorem}). Note that we can take $u_k^j$, for $j=1,2$, to be solutions to the same equation 
\[\sum_{i=1}^nD_i\left(\frac{D_iu^j_k}{\sqrt{1+|Du^j_k|^2}}\right)=\sum_{i=1}^nD_if^i_k+H_k \text{  in  } \W_k\]
and their boundary values to satisfy
\[\phi_k^1(x')\le\phi_k^2(x')\,,\,\forall x'\in\partial\W_k.\]

The above equation satisfies the comparison principle and hence $u_k^1(x')\le u_k^2(x')$ for all $x'\in\W_k$. The lemma follows by letting $k\to \infty$.
\end{proof}

Finally we want to show (Theorem \ref{large class with discontinuities}) that in case $n=2$, for a large class of boundary data $\W,\Phi$ the trace of the minimizer of the functional $\mathcal{F}$ with
$H$ equal to a non-negative constant has a jump discontinuity at a point where the mean curvature of $\partial\W$ is less than $-H$ and along this discontinuity it attaches to the prescribed boundary in a subset with non-empty interior (relative to the boundary manifold). For this we will need the following lemma.
\begin{lemma} \label{discontinuities thm}
Let $\W,\Phi,H$ be as in \emph{Theorem \ref{main theorem}}. Suppose that
$\{\Phi_{t}\}_{t\in [0,1]}$ is a continuous (as map from
$[0,1]$ into the space of $C^{1,\a}$ manifolds) $1$-parameter family, where for each $t\in[0,1]$, $\Phi_t$ is the limit of $C^{1,\a}$ graphs above $\partial\W$ and such that it satisfies the following: 

There exists $x_0'\in\partial\W$ and $\s>0$ such that:
\begin{enumerate}
\item[{i.}] $\Phi_{t}\cap (B^{n}_{\sigma}(x'_{0})\times\R)=\Phi\cap (B^{n}_{\sigma}(x'_{0})\times\R)$ for all $t\in
[0,1]$.
\item[{ii.}] $\{(x',\trace u_{0}(x')):x'\in\partial\Omega\cap B^{n}_\s(x'_0)\}\subset
V_{\Phi}$.
\item[{iii.}] $\{(x',\trace u_{1}(x')):x'\in\partial\Omega\cap B^{n}_\s(x'_0)\}\subset
U_{\Phi}$.
\end{enumerate}  
Here for each $t\in[0,1]$, $u_t\in BV(\W)$ is a minimizing function of the functional $\mathcal{F}$ with given data $(\W,\Phi_t,H)$.

Then for each $x_1\in \Phi\cap (B^{n}_{\sigma}(x'_{0})\times\R)$ there exists 
$t=t(x_1)\in (0,1)$ and $\e>0$ with $\{(x',\trace u_{t}(x')):x'\in\partial\Omega\}\cap
B^{n+1}_{\e}(x_1)=\Phi\cap B^{n+1}_{\e}(x_1)$.
\end{lemma}
\begin{proof}
 For each $t$, $(\partial\W\times\R)\setminus \trace u_t$ is the
  union of two disjoint connected components $U_t$, $V_t$, where $U_t\supset \{(x',x_{n+1}):x'\in
\partial\Omega,x_{n+1}>R\}$ and $V_{t}\supset \{(x',x_{n+1}):x'\in
\partial\Omega,x_{n+1}<-R\}$ for sufficiently large $R$. Given $x_1\in \Phi\cap (B^{n}_{\sigma}(x'_{0})\times\R)$, let
\[t_1=\sup\{t\in[0,1]:x_1\in U_t\}\]
and
\[t_2=\inf\{t\in[0,1]:x_1\in V_t\}.\]

Note that by the assumptions ii, iii and because of Remark \ref{uniqueness rmk}, $t_1, t_2\in (0,1)$. Take any sequence
$t_i\downarrow t_2$. Then by ii of Corollary \ref{corollary about trace},
$\vec{T}(\Omega,\Phi_{t_i})(x_1)$ coincides with the outward pointing normal of $\partial
\Omega\times\R$ at $x_1$ for all $t_i$ in the sequence and therefore, by Remark \ref{uniqueness rmk}, it is also true for $t_2$. On the other hand if we take a sequence $t_i\uparrow t_1$, then similarly we get that
$\vec{T}(\Omega,\Phi_{t_1})(x_1)$ coincides with the inward pointing normal of $\partial
\Omega\times\R$ at $x_1$.

Hence, again by Remark \ref{uniqueness rmk}, for some $t$ between $t_1$ and $t_2$, $\vec{T}(\Omega,\Phi_{t})(x_1)$ is not parallel to the normal vector to $\partial\W\times\R$ at $x_1$ and so for this $t$, Corollary \ref{corollary about trace} implies that $\{(x',\trace u_{t}(x')):x'\in\partial\Omega\}\cap
B^{n+1}_{\e}(x_1)=\Phi\cap B^{n+1}_{\e}(x_1)$ for some $\e>0$. \end{proof}

\begin{theorem}\label{large class with discontinuities}
Let $\W$ be a bounded $C^{2}$ domain of $\R^2$, $H\ge 0$ a given constant and $x'_{0}\in \partial\Omega$ is such that
\begin{equation*}\label{cond}
H_{\partial\W}(x'_0)<-H
\tag{(1)}
\end{equation*}
 where $H_{\partial\W}(x_0')$ denotes the mean curvature of $\partial\W$ with respect to the inward pointing unit
normal. 

Then there exists a large
class of $C^{1,\a}$ boundary data $\Phi$, for which the function $u$ that minimizes
the functional $\mathcal{F}$ with given data $\{\W,\Phi,H\}$ has trace with a jump discontinuity at $x'_0$ along which it attaches to $\Phi$ in a subset with non-empty interior.
\end{theorem}
\begin{proof}
Let $\Phi$ be an embedded $C^{1,\a}$ submanifold of $\partial\W\times\R$ such that for a sequence $\phi_i\in C^{1,\a}(\partial\W)$, $\graph \phi_i\toa \Phi$ and assume that
\[\{x_0'\}\times I\subset\Phi\]
for some interval $I$.

We will show that for any such $\Phi$ there exist $C^{1,\a}$ boundary data $\wt{\Phi}$ such that $\wt\Phi\cap (B_\s(x_0')\times\R)=\Phi\cap (B_\s(x_0')\times\R)$ for some $\s>0$ and for which the conclusion of the theorem holds, i.e. the function $u$ that minimizes
the functional $\mathcal{F}$ with given data $\{\W,\wt\Phi,H\}$ has trace with a jump discontinuity at $x'_0$ along which it attaches to $\wt\Phi$ in a subset with non-empty interior. Here and in the rest of the proof $B_r(x')$ will denote the 2-dimensional ball of radius $r$ centered at $x'$. By Lemma \ref{discontinuities thm}, it suffices to show that for such $\Phi$ we can construct a continuous 1-parameter family of boundary data $\{\Phi_t\}$ satisfying properties i-iii of Lemma \ref{discontinuities thm}.

For $\Phi$ as above and $\s>0$, that will be determined later, let $\{\Phi_t\}_{-\infty<t<\infty}$ be any monotone, continuous, 1-parameter family of boundary data satisfying the following:
\[\Phi_t\cap ( B_\s(x'_0)\times \R)=\Phi\cap( B_\s(x'_0)\times \R)\] 
and outside $B_\s(x'_0)\times \R$, $\Phi_t$ is given by the graph of a $C^{1,\a}$ function $\phi_t$ with
$$\lim_{t\rightarrow -\infty}\phi_t(x')=-\infty\text{  ,  }
\lim_{t\rightarrow +\infty}\phi_t(x')=\infty\quad \forall x'\in
\partial\W\setminus B_\s(x_0').$$
We will show that for $t_0>0$ big enough $\{\Phi_t\}_{\{-t_0\le t\le t_0\}}$ (after a reparametrization) satisfies properties ii and iii of Lemma \ref{discontinuities thm}. In particular we will show that there exists $t_0>0$ and $\s>0$, such that for all $t\ge t_0$
\begin{equation*}\label{firstinc}
 \{(x',\trace u_t(x')):x'\in\partial\W\cap B_\s(x_0')\}\subset U_{\Phi_t}
 \tag{(2)}
 \end{equation*}
 and
 \begin{equation*}\label{secondinc}
  \{(x',\trace u_{-t}(x')):x'\in\partial\W\cap B_\s(x_0')\}\subset V_{\Phi_{-t}}
   \tag{(3)}
\end{equation*}
 where $u_t$ is the minimizer of $\mathcal{F}$ with data $(\W,\Phi_t,H)$.
 
 For any $x'\in\partial\W$ such that $H_{\partial\W}(x')<0$, there exists a circumference $\mathcal C_{x'}$ passing through $x'$, such that a neighborhood of $x'$ in $\partial\W$ lies inside $\mathcal C_{x'}$.  Since $H_{\partial\W}(x'_0)<0$ we can choose $\s>0$ such that the following holds: 
For all $x'\in B_\s(x'_0)$ there exists a circumference $\mathcal C_{x'}$ passing through $x'$ and such that $B_\s(x'_0)\cap\partial\W$ lies inside $\mathcal C_{x'}$.

Let $x=(x',x_3)\in \Phi\cap (B^n_\s(x_0')\times\R)$, $\mathcal C_{x'}$ be as described above and let $\Delta$ be the region defined as follows: If $H=0$ then $\Delta$ is the region of $\R^2$ outside $\mathcal C_{x'}$ and if $H>0$ then $\Delta$ is an annulus with inner boundary $\mathcal C_{x'}$ and width $H^{-1}$. Then (cf. \cite{FinnBarriers}) there exist functions $v^\pm$ defined in $\Delta\cap\W$ such that $x\in\graph v^\pm$,
\[\sum_{i=1,2}D_i\left(\frac{D_iv^\pm}{\sqrt{1+|Dv^\pm|^2}}\right)= \pm H\]
and
 $Dv^\pm=\pm\infty$ on $\partial(\Delta\cap\W)\setminus\partial\W$. Taking $t$ small enough we have that
\[v^+\ge u_t\,\,\text{on}\,\,\partial\W\cap \Delta\]
and taking $t$ big enough we have that
\[v^-\le u_t\,\,\text{on}\,\,\partial\W\cap \Delta.\]

Hence we can apply the comparison principle in \cite{FinnBarriers} to conclude that for $t$ small enough, $v^+\ge u_t$ on $\Delta\cap\W$ and so $x$ lies above the trace of $u_t$. Similarly for $t$ big enough, $v^-\le u_t$ on $\Delta\cap\W$ and so $x$ lies below the trace of $u_t$.
 
\end{proof}

\begin{rmk}
By \emph{Remark \ref{uniqueness rmk}} and \emph{Corollary \ref{corollary about trace}}, \emph{Theorem \ref{large class with discontinuities}} still holds if $\W$ is a $C^{1,\a}$ domain. In this case condition \emph{\ref{cond}} of the theorem should be replaced by the following:
There exists some $r>0$ for which
\[-\int_{\partial\W} \nu_{\partial\W}\cdot D\zeta<\int_{\partial\W}H\zeta\]
for all positive $\zeta\in C^\infty$ with support in $B_r(x_0')\cap\partial\W$.
\end{rmk}

\section{Acknowledgment}
This work was mostly carried out as part of my thesis at Stanford University. I would like to thank my advisor Leon Simon, to whom I am indebted for all his help and guidance.

\appendix

\section{ A Technical Lemma}
\begin{lemma}\label{match divergence}
Let $\W$ be a $C^{1,\a}$ domain of $\R^n$ and $r>0$ be such that 
\[\k(\partial\W,r,x)<1\,\,,\,\,\forall x\in\partial\W\]
with $\k$ as defined in \emph{Definition \ref{c1a convergence}}.

Given $\eta=(\eta_1,\dots,\eta_n)\in C^{0,\a}(\partial\W;\R^n)$, there exists a $C^{0,\a}$ vector field $X$ on ${\W}^{r/4}=\{x\in\W:0\le\dist(x,\partial\W)\le r/4\}$ such that $\dvg X=0$ (weakly), $X|_{\partial\W}=\eta$ and for any $x\in\partial\W$
\[\|X\|_{0, B^n_r(x)\cap \W^{r/4}}+r^{\a}[X]_{\a, B^n_r(x)\cap\W^{r/4}}\le C\left(\|\eta\|_{0, B^n_r(x)\cap\partial\W}+r^{\a}[\eta]_{\a, B^n_r(x)\cap\partial\W}\right)\]
where $C$ depends on $\partial\W$ and $n$.
\end{lemma}
\begin{proof}

We will construct $X$ around each point using local transformations that flatten the boundary. Therefore we will first show that the lemma holds in the case of flat boundary:

\textit{Claim:} Given a $C^{0,\a}$ vector field $g=(g_1,\dots, g_{n}):\R^{n-1}\to\R^n$ with compact support there exists a $C^{0,\a}$ vector field $X=(X_1,\dots,X_n)$ on $\R^{n-1}\times[0,1]$ such that $X(x',0)=g(x')$, $\dvg X=0$ (weakly) on $\R^{n-1}\times[0,1]$ and $\|X\|_{0,\a}\le C\|g\|_{0,\a}$, where $C$ is an absolute constant.

Let $\phi$ be a non-negative, smooth function with compact support in $B^{n-1}_1(0)$ and such that $\int_{\R^{n-1}}\phi(\xi)d\xi=1$. For $x=(x',x_n)$ define $X$ by the following formula:
\begin{equation*}\label{X}
\begin{split}X(x',x_n)=\sum_{j=1}^{n-1}&\int_{\R^{n-1}}g_j(x'-x_n\xi)\left(\phi(\xi)-\dvg(\phi(\xi)\xi)\right)d\xi \,e_j\\
&+\left(g_n(x')-\sum_{j=1}^{n-1}\int_{\R^{n-1}}g_j(x'-x_n\xi) D_j\phi(\xi)d\xi\right)e_n\end{split}
\tag{(1)}
\end{equation*}
For $X$ it is easy to check that $X(x',0)=g(x')$ and $\|X\|_{0,\a}\le C\|g\|_{0,\a}$. Furthermore we have that $\dvg X=0$. To see this, we only need to check it for smooth $g$, for which we can integrate by parts to get that
\[X(x',x_n)=\sum_{j=1}^{n-1}D_nf_j(x) e_j+\left(g_n(x')-\sum_{j=1}^{n-1}D_jf_j(x)\right) e_n\]
where
\[f_j(x',x_n)=x_n\int_{\R^{n-1}}g_j(x'-\xi x_n)\phi(\xi)d\xi.\]
Then
\[\dvg X=\sum_{j=1}^{n-1}D_jD_nf_j+D_ng_n-\sum_{j=1}^{n-1}D_nD_jf_j=0\]
since $g_n$ is independent of the $x_n$-variable. Hence the claim is true.

For the general case, consider a finite cover $\{B^{n}_{r/4}(x_i)\}$ of $\partial\W$, where $x_i\in\partial\W$ and such that $B^n_{r/8}(x_i)\cap B^n_{r/8}(x_j)=\emptyset$ for $x_i\ne x_j$. Let $\phi_i$ be a partition of unity suboordinate to this cover. Let $\eta_i(x)=\phi_i(x)\eta(x)$, for all $x\in\partial\W$.

For each $i$, let $\psi_i$ be a diffeomorphism that flattens $\partial\W$ in $B^{n}_{r}(x_i)$, i.e. 
\[\psi_i(\W\cap B^{n}_{r}(x_i))=B^{n}_{r}(0)\cap \R^n_+.\] 
We can also take $\psi_i$ so that $\psi_i(\W\cap B^{n}_{r/4}(x_i))=B^{n}_{r/4}(0)\cap \R^n_+$ and  
\[\psi_i(\W^{r/4}\cap B_r(x_i))\subset B^n_r(0)\cap(\R^{n-1}\times[0,r/4]).\] Let
\begin{equation*}\label{g}
g_i(x)=(D_{\psi_i^{-1}(x)}\psi_i)\eta_i(\psi_i^{-1}(x))\,,\,\text{ for }x\in B^n_r(0)\cap(\R^{n-1}\times\{0\})
\tag{(2)}
\end{equation*}
where $(D_{\psi_i^{-1}(x)}\psi_i)$ denotes the matrix of the Jacobian of $\psi_i$ at the point $\psi_i^{-1}(x)$.

Then, since $\psi_i(\W\cap B^{n}_{r/4}(x_i))=B^{n}_{r/4}(0)\cap \R^n_+$ and $\eta_i=0$ outside $B_{r/4}^n(x_i)$ we have that 
\begin{equation*}\label{gj zero}
g_i(x)=0\text{ for }x\in(B^n_r(0)\cap(\R^{n-1}\times\{0\}))\setminus B^{n}_{r/4}(0).
\tag{(3)}
\end{equation*}
For each $g_i$ let $X_i$ be the vector field given by \ref{X}. Then, by \ref{gj zero}
\begin{equation*}\label{x zero}
X_i=0\text{ on }\partial B^{n}_{r}(0)\cap(\R^{n-1}\times[0,r/4]). 
\tag{(4)}
\end{equation*}
Hence
\[X(x)=\sum_i (D_{\psi_i(x)}{\psi_i}^{-1})X_i(\psi_i(x))\]
is a $C^{0,\a}$ vector field on $\W^{r/4}$ and is such that for any $x\in\partial\W$
\[\begin{split}\|X\|_{0,\W^{r/4}\cap B^n_r(x)}&\le \sum_i\|D{\psi^{-1}_i}\|_{0, B^n_r(0)}\|X_i\|_{0, \psi_i(B^n_r(x))}\\&\le C\sum_i\|g_i\|_{0,\psi_i(B^n_r(x))\cap(\R^{n-1}\times\{0\})}\le C\|\eta\|_{0,\partial\W\cap B^n_r(x)}\end{split}\]
and
\[\begin{split}r^\a[X]_{\a,\W^{r/4}\cap B^n_r(x)}\le&\sum_ir^\a[D{\psi^{-1}_i}]_{\a, B^n_r(0)}\|X_i\|_{0, \psi_i(B^n_r(x))}\\&+\|D{\psi^{-1}_i}\|_{0, B^n_r(0)}r^\a[X_i]_{\a, \psi_i(B^n_r(x))}\\
\le& C\sum_i\|g_i\|_{0,\psi_i(B^n_r(x))\cap(\R^{n-1}\times\{0\})}+r^\a[g_i]_{\a,\psi_i(B^n_r(x))\cap(\R^{n-1}\times\{0\})}\\
&\le C\left(\|\eta\|_{0,B^n_r(x)\cap\partial\W}+r^\a[\eta]_{\a,B^n_r(x)\cap\partial\W}\right)\end{split}\]
where $C$ depends on $n$. 
 
Finally we have that $X|_{\partial\W}=\eta$ and $\dvg X=0$. To see this let $x\in \partial\W$, then:
\[\begin{split}X(x)=&\sum_i (D_{\psi_i(x)}{\psi_i}^{-1})g_i(\psi_i(x))=\sum_i(D_{\psi_i(x)}{\psi_i}^{-1}) (D_{x}{\psi_i})\eta_i(x)\\
=&\sum_i \phi_i(x)\eta(x)=\eta(x)\end{split}\]

Let $\zeta\in C^\infty(\W^{r/4})$ and having compact support. We will show that
\[\int_{\W^{r/4}} X(x)\cdot D_x\zeta d\H^{n}(x)=0.\]

\[\begin{split}\int_{\W^{r/4}}X(x)\cdot D_x&\zeta d\H^{n}(x)\\
&=\sum_i\int_{B^n_r(x_i)}(D_{\psi_i(x)}{\psi_i}^{-1})X_i(\psi_i(x))\cdot D_x\zeta d\H^{n}(x)\\
&=\sum_i\int_{B^n_r(x_i)}(D_{\psi_i(x)}{\psi_i}^{-1})X_i(\psi_i(x))\cdot D_{\psi(x)}(\zeta\circ\psi^{-1})D_x\psi d\H^{n}(x)\\
&=\sum_i\int_{B^n_r(x_i)}X_i(\psi_i(x))\cdot D_{\psi(x)}(\zeta\circ\psi^{-1}) d\H^{n}(x)\\
&=\sum_i\int_{B^n_r(0)}X_i(y)\cdot D_y(\zeta\circ\psi^{-1}) d\H^n(y)\end{split}\]
which is equal to zero because by construction $\dvg X_i=0$ (weakly).
\end{proof}

\begin{rmk}\label{rmk match div}
We remark that lemma \emph{\ref{match divergence}} is true with higher regularity of the boundary and of the vector field $\eta$. In particular we have the following:

Let $\W$ be a $C^{k,\a}$ domain of $\R^n$ and $\eta\in C^{l,\beta}(\partial\W,\R^n)$, where $k,l\ge 0$ , $\a,\beta\in[0,1]$ and $l+\beta\le k+\a+1$.  Then there exist a neighborhood $V$ of $\partial\W$ in $\W$ and a $C^{l,\beta}$ vector field $X$ on $V$ such that $\dvg X=0$, $X|_{\partial\W}=\eta$ and
\[\|X\|_{l,\beta,V}\le C\|\eta\|_{l,\beta,\partial\W}\]
where the neighborhood $V$ and the constant $C$ depend on $\partial\W$.
\end{rmk}

\section{Varifolds, Currents}
Let $$P:\R^{n+1}\rightarrow \R^2\times\{0\}^{n-1}$$ denote the projection onto the
($x_1,x_2$)-plane in $\R^{n+1}$.\\
Let
$$\zeta:\R^{n+1}\setminus(\{0\}^2\times\R^{n-1}) \rightarrow \R^2\times\{0\}^{n-1}$$
be the function defined by
$$\zeta(x)=\zeta(x_1,x_2,\dots,x_{n+1})=(-x_2,x_1,0,\dots,0)$$
so that $\zeta$ is the projection $P$ followed by a counterclockwise
$\pi/2$-rotation.

\begin{lemma}\label{cone as a union of planes}\emph{(Allard)}\\
Let $C=(\spt C,\vec{C},\theta)$ be an
$n$-dimensional cone in $R^{n+1}$, such that $0\in
\partial C$, $\theta(x)\ge1$ for all $x\in \spt C\setminus \partial C$,
  $\partial C=\{0\}^2\times \R^{n-1}$ and $\|\d C\|(\R^{n+1}\setminus \partial C)=0$.
  
 For each $\phi\in C^\infty((R^2\times\{0\}^{n-1})\cap S^n)$ define
 $$T(\phi)=\int_{B^{n+1}_1(0
 )\setminus \partial C}\phi\left(\frac{P(x)}{|P(x)|}\right)
 \frac{|p_{x}(\zeta(x))|^2}{|P(x)|^2}\theta(x)d\H^n(x)$$
 where $\zeta$ is as defined above and $p_x$ denotes the projection onto
 the tangent space of $C$ at $x$.
 
 Then \begin{enumerate}
        \item $T$ is a multiple of $\H^1((R^2\times\{0\}^{n-1})\cap S^n)$, i.e.
\begin{equation*}
  T(\phi)=c\int_{(\R^2\times\{0\}^{n-1})\cap S^n}\phi(x)d\H^1(x)
\end{equation*}
for any $\phi\in C^\infty((R^2\times\{0\}^{n-1})\cap S^n)$.
        \item If $T=0$ then $P(\spt C)\cap S^n$ is finite.
      \end{enumerate}
\end{lemma}
For the proof of this lemma we refer to \cite{Abdryregularity}.

Part (2) of Lemma \ref{cone as a union of planes} directly implies the following:
\begin{cor}\label{corollary cone union of planes}
Let $C$ be an $n$-dimensional cone in $\R^{n+1}$ such that $0\in\partial C$, $\partial C$ is an $(n-1)$-dimensional subspace, $\theta(x)\ge 1$ for all $x\in\spt C\setminus\partial C $ and $\spt C\subset H$, where $H$ is a halfspace with $\partial C\subset \partial H$.

Then
\[C=\sum_{i=1}^k P_i\]
where $P_i$ are $n$-dimensional halfspaces, with $\partial P_i=\pm\partial C$ and $P_i\subset H$.
\end{cor}

\begin{cor}\label{corollary cone in plane}
If in addition to the hypotheses of \emph{Corollary \ref{corollary cone union of planes}}, we assume that $C$ is area minimizing, then we have that either $C$ is an $n$-dimensional halfspace or
\[C=mH_1+lH_2\]
where $H_1$, $H_2$ are the two halfspaces in $\partial H$ defined by $\partial C$.

Furthermore $|m-l|$ gives the multiplicity of $\partial C$.
\end{cor}

\begin{proof}
Without loss of generality we can assume that $H=\ov{R}_+\times\R^n$.
By Corollary \ref{corollary cone union of planes} we can write $C=\sum_{i=1}^{k}P_i$, where $P_i$ are now of multiplicity 1 (so that we could
have that $P_i=P_j$),
$$P_i=\pm[\![\{y+tu_i, y\in\partial C, t>0\}]\!]$$
for some unit vector $u_i$, normal to $\partial C$ and such that $u_i\cdot e_1\ge 0$.

Take $j\in\{1,\dots k\}$ such that $\partial P_j=\partial C$, then
$$\partial(C-P_j)=0$$
and for any compact set $W\subset\R^{n+1}$
$$\underline{\underline{M}}_W(C-P_j)=
\underline{\underline{M}}_W(C)-\underline{\underline{M}}_W(P_j).$$ 

We claim that $C-H_j$ is also area minimizing. Assume that it is not true, then for a compact set $W\subset \R^{n+1}$  there exists a current $S$ with $\spt S\subset W$,
$\partial S=0$ and such that
$$\underline{\underline{M}}_W(C-P_j+S)<\underline{\underline{M}}_W(C-P_j)=
\underline{\underline{M}}_W(C)-\underline{\underline{M}}_W(P_j).$$  
Then
$$\underline{\underline{M}}_W(C+S)\le \underline{\underline{M}}_W(C+S-P_j)+
\underline{\underline{M}}_W(P_j)<\underline{\underline{M}}_W(C)$$ which contradicts the
fact that $C$ is area minimizing.

So $C-P_j$ is area minimizing and hence the associated varifold is stationary. Computing
$\|\d(C-P_j)\|(B_R(0))$ we get that
$$0=\sum_{\substack{i=1\\i\ne j}}
^{k}u_i.$$ This is true for any $j$ such that $\partial P_j=\partial C$, hence
there can only
be one such different $u_j$.

Similarly, picking a $j$ such that $\partial P_j=-\partial C$ we get that $C+H_j$ will
be area minimizing and computing the first variation of the corresponding varifold we
get that
$$0=u_j+\sum_{i=1}^{k}u_i.$$
Hence, as before, there can only
be one such different $u_j$.

So either $k=1$, in which case we get that $C$ is an $n$-dimensional halfspace or if $k>1$ we showed that $C$ must be of the form
$$C=kH_1+lH_2$$
where $H_1, H_2$ are the two halfspaces in $\partial H$ defined by $\partial C$.
\end{proof}

\begin{lemma}\label{area minimizing in halfspace}
 Let $C$ be an $n$-dimensional integral current such that $\spt C$ lies in a closed halfspace $H$, $\partial C\subset\partial H$ and $C$ minimizes area in $H$. Then $C$ is area minimizing.
 \end{lemma}
\begin{proof}
Suppose not. Then there exists an integer multiplicity current $S$, with $\partial S=
\partial C$, $W=\spt(S-C)$ compact in $\R^{n+1}$ and such that
\begin{equation*}\label{arg}
 \mass_W(C)>\mass_W(S).
\tag{(1)}
\end{equation*}
Let $f$ be the reflection along $L=\partial H$: $$f(x)=L(x)-L^\bot(x)\text{ , }x\in \R^{n+1}$$
where, for a subspace $P$, $P(x)$ denotes the projection of $x$ on $P$.

Define the function $g:\R^{n+1}\rightarrow H$, by:
$$g(x)=\begin{cases} x&, x\in H \\ f(x)&, x\in \R ^{n+1}\setminus H.
\end{cases}$$
Then for the current $g_\# S$ we have that it has support in $H$, $\partial( g_\# S) =g_\#\partial S= \partial C$ and it satisfies the estimate
\begin{equation*}\label{shest}
\mass_V(g_\# S)\le\sup_{g^{-1}(V)}|Dg|^n\mass_{g^{-1}(V)}(S)\,\,,\,\,\forall\,V\subset\subset H \tag{(2)}
\end{equation*}
Where, if $S=(\spt S, \theta,
\vec{S})$, then:
\begin{align*}
\label{gsharp} (g_\# S) (\omega)= \int_{\spt S}<\omega(g(x)),
dg_{x\#}\vec{S}(x)>\theta(x) d\mathcal{H}^n(x).
\end{align*}

Using now the assumption on $C$ and \ref{shest} we have that for any compact subset of $\R^{n+1}$, $W$:
\[\mass_W(C)=\mass_{W\cap H}(C)\le\mass_{W\cap H}(g_\# S)\le\mass_W(S)\]
which contradicts \ref{arg}.
\end{proof}

 \bibliographystyle{amsalpha}
\bibliography{bibliography}

\end{document}